\documentclass{article}
\usepackage{amsthm}
\usepackage{amsfonts}
\usepackage{amssymb}
\usepackage{amsmath}
\usepackage{graphicx}
\usepackage{tikz}
\usepackage{tikz-cd}
\usepackage{hyperref}
\usepackage[capitalise]{cleveref}
\usepackage{adjustbox}
\usepackage{rotating}

\newtheorem{theorem}{Theorem}[section]
\newtheorem{lemma}[theorem]{Lemma}
\newtheorem{proposition}[theorem]{Proposition}
\newtheorem{corollary}[theorem]{Corollary}

\theoremstyle{definition}
\newtheorem{definition}[theorem]{Definition}
\newtheorem{example}[theorem]{Example}
\newtheorem*{definition*}{Definition}

\theoremstyle{remark}
\newtheorem{remark}[theorem]{Remark}

\newcommand{\B}{\mathcal B}

\renewcommand{\H}{\mathcal H}

\newcommand{\K}{\mathcal K}
\renewcommand{\L}{\mathcal L}
\newcommand{\M}{\mathcal M}
\newcommand{\N}{\mathcal N}
\renewcommand{\O}{\mathcal O}

\newcommand{\CC}{\mathbb C}

\newcommand{\NN}{\mathbb N}

\newcommand{\RR}{\mathbb R}
\newcommand{\TT}{\mathbb T}

\newcommand{\ZZ}{\mathbb Z}
\newcommand{\Aut}{\mathrm{Aut}}
\newcommand{\Spec}{\mathrm{Spec}}
\newcommand{\Stab}{\mathrm{Stab}}
\newcommand{\iso}{\cong}
\newcommand{\diag}{\mathrm{diag}}
\newcommand{\id}{\mathrm{id}}
\newcommand{\rank}{\mathrm{rank}}
\newcommand{\cat}{\mathsf}
\newcommand{\cns}{\mathrm}
\renewcommand{\:}{\colon}
\newcommand{\tr}{\mathrm{tr}}
\newcommand{\tensor}{\mathbin{\otimes}}
\newcommand{\Ad}{\mathrm{Ad}}
\newcommand{\VT}{\mathrm{VT}}

\renewcommand{\>}{\rangle}

\title{Vertex-transitive quantum graphs}
\author{Mac Hayes, Trevor Jess, Andre Kornell, and Remi Salinas Schmeis}
\date{}

\newcommand{\Addresses}{{
  \bigskip
  \footnotesize

  \textsc{Department of Mathematics, University of Kansas}
  \par\nopagebreak
  \textsc{Lawrence, KS 66045}
  \par\nopagebreak
  \textit{E-mail address}: \texttt{machayes00@ku.edu}

  \medskip
  
  \textsc{Department of Mathematical Sciences, New Mexico State University}
  \par\nopagebreak
  \textsc{Las Cruces, NM 88003}
  \par\nopagebreak
  \textit{E-mail address}: \texttt{tkjess@nmsu.edu}

    \medskip
  
  \textsc{Department of Mathematical Sciences, New Mexico State University}
  \par\nopagebreak
  \textsc{Las Cruces, NM 88003}
  \par\nopagebreak
  \textit{E-mail address}: \texttt{kornell@nmsu.edu}

    \medskip
  
  \textsc{Department of Mathematical Sciences, New Mexico State University}
  \par\nopagebreak
  \textsc{Las Cruces, NM 88003}
  \par\nopagebreak
  \textit{E-mail address}: \texttt{remis@nmsu.edu}

}}

\allowdisplaybreaks

\begin{document}

\maketitle

\begin{abstract}
We define a quantum graph to be vertex-transitive if the join of its automorphism group is the maximum quantum relation on its quantum vertex set, in direct analogy with the classical case. All simple quantum graphs in $M_2(\CC)$ are vertex-transitive, but many simple quantum graphs in $M_3(\CC)$ are not vertex-transitive. We provide a complete classification of vertex-transitive quantum graphs in $M_3(\CC)$ up to isomorphism. To do this, we introduce a polynomial invariant for quantum graphs in $M_n(\CC)$, which we call the panoramic polynomial.
\end{abstract}

\vspace{.5cm}


\vspace{.5cm}


\section{Introduction}

Quantum graphs were introduced by Duan, Severini, and Winter in \cite[sec.~II]{zbMATH06727895}, where a quantum graph is synonymously an operator system of matrices, i.e., a subspace $R \subseteq M_n(\CC)$ that is closed under the Hermitian adjoint operation and contains the identity matrix \cite[p.~157]{zbMATH03532026}. Operator systems had been studied for decades, but this analogy between operator systems and graphs was new. Specifically, the authors associated a ``confusability'' quantum graph with each quantum channel, obtaining a formula for the zero-error capacity of the quantum channel that is clearly analogous to the classical formula \cite[sec.~III]{zbMATH06727895}.

This quantum generalization of graphs has many variants, e.g., \cite{zbMATH06008057, zbMATH06936038, zbMATH07202497, zbMATH07856619}, which support a growing literature on quantum graphs, e.g., \cite{zbMATH07502493,zbMATH07809225, zbMATH07895052, Goldberg2026, arXiv:2601.09685, arXiv:2511.23121}.

\noindent Operator systems of matrices are the basic examples of quantum graphs for each of these variants. Operator systems in $M_2(\CC)$ are classified by \cite[Theorem~3.11]{zbMATH07632578} and \cite[Theorem~3.1]{zbMATH07668023}, but we are not aware of any classification results for operator systems in $M_3(\CC)$ prior to the present work.

Simple quantum graphs in $M_3(\CC)$, which we directly investigate, are exactly the orthocomplements of these operator systems. We omit the word ``simple,'' referring to simple quantum graphs as \emph{quantum graphs} and to simple graphs as \emph{graphs}, following a standard convention of graph theory.

The primary significance of a classification result for quantum graphs in $M_3(\CC)$ is that it reveals a natural class of structurally transparent examples. Classical graph theory is founded on a menagerie of small graphs that gives graph theorists their basic intuitions \cite{zbMATH00854567}, and quantum graph theory needs small examples for the same purpose. Quantum graphs in $M_2(\CC)$ remain the initial reference point. However, $M_2(\CC)$ is often regarded as being pathologically small in the context of quantum theory, e.g., \cite{zbMATH03128586,zbMATH00539968,zbMATH01477902}.

We classify the vertex-transitive quantum graphs in $M_3(\CC)$ (see Figure~\ref{fig. introduction}). All quantum graphs in $M_2(\CC)$ are vertex-transitive by \cite[Theorem~4.2]{zbMATH07668023}, but almost all quantum graphs in $M_3(\CC)$ are not vertex-transitive by \cite[Theorem~C]{zbMATH07502493}.

\begin{definition*}
Let $R \subseteq M_n(\CC)$ be a subspace. Then,
\begin{enumerate}
\item $R$ is a \emph{quantum graph} if $a^\dagger \in R$ and $\cns{Tr}(a) = 0$ for all $a \in R$,
\item $R$ is \emph{vertex-transitive} if furthermore $\{u \in U(n) \mid u R u^\dagger = R\}$ spans $M_n(\CC)$.
\end{enumerate}
\end{definition*}

\begin{figure}[ht]
\centering
\begin{tabular}{|c|c|c|c|c|} 
 \hline
 $R$ & $n$ & $\dim(R)$ & $\Aut(R)$ & $\cns{diam}(R)$ \\
 \hline
 $\mathrm{VT}^1_0$ & 1 & 0 & 1 & $0$ \\
 \hline
 $\mathrm{VT}^2_0$ & 2 & 0 & $SO(3)$ & $\infty$ \\
 $\mathrm{VT}^2_1$ & 2 & 1 & $O(2)$ & $\infty$ \\ 
 $\mathrm{VT}^2_2$ & 2 & 2 & $O(2)$  & $2$  \\
 $\mathrm{VT}^2_3$ & 2 & 3 & $SO(3)$ & $1$  \\ 
 \hline
 $\mathrm{VT}^3_0$ & 3 & 0 & $PU(3)$ & $\infty$ \\
 $\mathrm{VT}^3_2$ & 3 & 2 & $\TT^2 \rtimes S_3$ & $\infty$ \\ 
 $\mathrm{VT}^3_3(0)$ & 3 & 3 & $S_4$  & $2$  \\
 $\mathrm{VT}^3_3(\theta)$ & 3 & 3 & $A_4$  & $2$  \\
 $\mathrm{VT}^3_3(\frac \pi 2)$ & 3 & 3 & $SO(3)$  & $2$  \\
 $\mathrm{VT}^3_4$ & 3 & 4 & $(S_3 \wr S_2) \cap A_6$ & $2$  \\
 $\mathrm{VT}^3_5(0)$ & 3 & 5 & $SO(3)$  & $2$  \\
 $\VT^3_5(\theta)$ & 3 & 5 & $A_4$  & $2$  \\
 $\VT^3_5(\frac \pi 2)$ & 3 & 5 & $S_4$ & $2$ \\
 $\VT^3_6$ & 3 & 6 & $\TT^2 \rtimes S_3$ & $2$ \\
 $\VT^3_8$ & 3 & 8 & $PU(3)$ & $1$ \\ 
 \hline
\end{tabular}
\caption{All vertex-transitive quantum graphs $R \subseteq M_n(\CC)$ for $n \leq 3$ modulo isomorphism, i.e., unitary equivalence. The parameter $\theta$ takes values in  $(0, \frac \pi 2)$.}
\label{fig. introduction}
\end{figure}

These definitions, as well as the definitions of the graph parameters that appear in Figure~\ref{fig. introduction}, have a natural basis in the noncommutative metaphor of noncommutative geometry \cite{zbMATH01452515,zbMATH07287276}. These motivations are discussed at length in section~\ref{sec. quantum graphs}. To obtain this classification, we introduce a polynomial invariant that we call the panoramic polynomial for quantum graphs $R \subseteq M_n(\CC)$ in section~\ref{sec. panoramic polynomial}. Thus, the primary contributions of this work are
\begin{enumerate}
\item a natural notion of vertex-transitivity for quantum graphs,
\item a useful polynomial invariant for quantum graphs in $M_n(\CC)$,
\item a complete classification of vertex-transitive quantum graphs in $M_3(\CC)$.
\end{enumerate}

The vertex-transitive quantum graphs in $M_3(\CC)$ modulo isomorphism are given in Figure~\ref{fig. classification} together with their panoramic polynomials. The panoramic polynomials for the quantum graphs $\VT^3_d$ with $d > 4$ are provided for completeness only. Indeed, the panoramic polynomial of a quantum graph $R \subseteq M_n(\CC)$ has $\dim(R)$ variables, so we use this invariant to classify the vertex-transitive quantum graphs $R \subseteq M_3(\CC)$ such that $\dim(R) \leq 4$ and appeal to \cref{complement} to complete the classification. We obtain a classification  of even those vertex-transitive quantum graphs that need not be simple; see \cref{rem. all}.

Readers who are familiar with quantum isomorphism \cite{zbMATH07020847} will quickly recognize that all of the quantum graphs in Figure~\ref{fig. classification} except $\VT^3_3(\theta)$ and $\VT^3_5(\theta)$, for $0 \leq \theta \leq \pi/2$,
are quantum isomorphic to graphs. However, we do not provide a proof of this claim to avoid reviewing and adapting this involved notion for the sake of a minor observation.

\bigskip

\noindent \textbf{Notation.} We write $R \leq M_n(\CC)$ to indicate that $R$ is a subspace of $M_n(\CC)$. For each positive integer $n$, $\tr\: M_n(\CC) \to \CC$ denotes the normalized trace, so $\tr(1_n) = 1$. The canonical inner-product on $M_n(\CC)$ is defined by $\<a | b\> = \tr(a^\dagger b)$, where $a^\dagger$ is the Hermitian adjoint of $a$. The \emph{norm} of a matrix $a \in M_n(\CC)$ is always $\|a\| = \sqrt{\tr(a^\dagger a)}$. In contrast, the Schatten $p$-norm of $a$ is notated $\|a\|_p$ for all $p \in [1, \infty]$.

\bigskip

\noindent \textbf{Acknowledgments.} We thank John Harding and Violetta Garcia for contributing key ideas to this project and Gabriel Agnew, Hannah Himelright, and Diego Langevin for useful discussion. We also thank Alexandru Chirvasitu and Mateusz Wasilewski for their comments. This research was supported by the National Science Foundation under Award No.~DMS-2231414.

\section{Quantum graphs}\label{sec. quantum graphs}

The present article deals exclusively with quantum graphs in $M_n(\CC)$. However, in this section, we motivate the definitions in section~\ref{sec. panoramic polynomial} in a broader context. Specifically, we motivate the definition of vertex-transitivity for the canonical quantum generalization of graphs in noncommutative geometry \cite{arXiv:2601.09685}.

For us, a \emph{graph} is a set that is equipped with an irreflexive symmetric binary relation. 
Thus, loops and multiple edges are forbidden. Both sets and binary relations have canonical quantum generalizations in noncommutative geometry. Sets, i.e., discrete quantum spaces, are generalized to $C^*$-algebras that are $c_0$-direct sums of full matrix algebras \cite{zbMATH04152742} or, equivalently, to von Neumann algebras that are $\ell^\infty$-direct sums of full matrix algebras \cite{zbMATH07287276}. These von Neumann algebras are called the \emph{hereditarily atomic} von Neumann algebras.
Binary relations are then generalized to quantum relations in the sense of Weaver \cite{zbMATH06008057}.

Sets and binary relations form a category $\cat{Rel}$, and similarly, hereditarily atomic von Neumann algebras and quantum relations form a category $\cat{qRel}$. Both $\cat{Rel}$ and $\cat{qRel}$ are dagger compact categories that are enriched over complete modular ortholattices \cite[sec.~3]{zbMATH07287276}. This categorical structure essentially consists of standard constructions for sets and binary relations and for their quantum generalizations. We do not review this categorical structure but instead make do with the fragment of the noncommutative dictionary in Figure~\ref{fig. noncommutative metaphor}.

\begin{figure}[ht]
\begin{center}
\begin{tabular}{| c |c|} 
 \hline
 classical concept & quantum generalization \\
 \hline
 set & hereditarily atomic von Neumann\ algebra\\
 binary relation & quantum relation \\
 diagonal relation &  commutant quantum relation \\
 converse of a binary relation & adjoint of a quantum relation \\
 composition of binary relations & product of quantum relations \\
 union of binary relations & join of quantum relations\\
 disjoint binary relations & orthogonal quantum relations \\
 \hline
\end{tabular}
\caption{A fragment of the noncommutative metaphor.}
\label{fig. noncommutative metaphor}
\end{center}
\end{figure}

\begin{definition}
We define the terms in Figure~\ref{fig. noncommutative metaphor}. Let $\M \subseteq \B(\H)$, $\N \subseteq \B(\K)$, and $\O \subseteq \B(\L)$ be von Neumann algebras.
\begin{enumerate}
\item The von Neumann algebra $\M$ is \emph{hereditarily atomic} \cite{zbMATH07287276} if it is of the form
$$
\M \iso \bigoplus_{i \in I}^{\ell^\infty} M_{n_i}(\CC).
$$
\item A \emph{quantum relation} \cite{zbMATH06008057} from $\M$ to $\N$ is an ultraweakly closed subspace $R \subseteq \B(\H, \K)$ such that $\N' R \M' \subseteq R$.
\item The \emph{commutant quantum relation} on $\M$ is the commutant $\M' \subseteq \B(\H, \H)$.
\item The \emph{adjoint} of a quantum relation $R$ from $\M$ to $\N$ is the quantum relation $R^\dagger$ from $\N$ to $\M$ that is defined by
$
R^\dagger = \{a^\dagger \mid a \in R\}.
$
\item The \emph{product} of quantum relations $R$ from $\M$ to $\N$ and $S$ from $\N$ to $\O$ is the quantum relation $SR$ from $\M$ to $\O$ that is defined by
$$
S R = \overline{\cns{span}\{b a \mid a \in R,\,b\in S\}}.
$$
\item The \emph{join} of quantum relations $R_i$ from $\M$ to $\N$, for $i \in I$, is their join with respect to subspace inclusion. Thus,
$$\bigvee_{i \in I} R_i = \overline{\sum_{i \in I} R_i}.$$
\item Two quantum relations $R_1$ and $R_2$ from $\M$ to $\N$ are \emph{orthogonal} if $\cns{Tr}(a_1^\dagger a_2)$ vanishes whenever $a_1 \in R_1$ and $a_2 \in R_2$ are Hilbert-Schmidt.
\end{enumerate}
\end{definition}

The fragment of the noncommutative metaphor in Figure~\ref{fig. noncommutative metaphor} immediately yields the following quantum generalizations.

\begin{definition}\label{defn. in qRel} Let $\M$ be a hereditarily atomic von Neumann algebra, and let $R$ be a quantum relation on $\M$.
\begin{enumerate}
\item The pair $(\M, R)$ is a \emph{quantum graph} if 
\begin{enumerate}
\item $R^\dagger = R$,
\item $R$ is orthogonal to $\M'$.
\end{enumerate}
\item A quantum graph $(\M, R)$ is \emph{vertex-transitive} if the join of all quantum relations $F$ on $\M$ such that
\begin{enumerate}
\item $F^\dagger F = \M' = F F^\dagger$, 
\item $F R = R F$
\end{enumerate}
is the maximum quantum relation on $\M$.
\end{enumerate}
\end{definition}

It is straightforward to check that \cref{defn. in qRel} reduces to Definitions \ref{defn. quantum graph} and \ref{defn. aut} when $\M = M_n(\CC)$.

\section{Panoramic polynomial}\label{sec. panoramic polynomial}

In this section, we define the basic terms of the article and define the panoramic polynomial, which is a polynomial invariant for quantum graphs.

\begin{definition}[\cite{zbMATH06727895} and \cite{zbMATH06936038}]\label{defn. quantum graph}
Let $R \leq M_n(\CC)$. Then,
\begin{enumerate}
\item $R$ is a \emph{quantum graph} if $\tr(a) = 0$ and $a^\dagger \in R$ for all $a \in R$,
\item an orthonormal basis of $R$ is \emph{Hermitian} if each basis element is Hermitian.
\end{enumerate}
\end{definition}

Every quantum graph has a Hermitian orthonormal basis, which may be obtained as an orthonormal basis for the real inner-product space $R_h$ that consists of the Hermitian matrices in $R$. More generally, every subspace $R$ that satisfies $R^\dagger = R$ has a Hermitian orthonormal basis. Conversely, if a subspace $R $ has a Hermitian orthonormal basis, then it immediately satisfies $R^\dagger = R$.

\begin{definition}\label{defn. iso}
Let $R, S \leq M_n(\CC)$ be quantum graphs. We define $R$ and $S$ to be \emph{isomorphic} and write $R \iso S$ if there exists a unitary $u \in M_n(\CC)$ such that $$u R u^\dagger = S.$$
In this case, the subspace $\CC u \leq M_n(\CC)$ is an \emph{isomorphism} from $R$ to $S$.
\end{definition}

\begin{definition}\label{defn. aut}
Let $R \leq M_n(\CC)$ be such that $R^\dagger = R$. Then,
\begin{enumerate}
\item the \emph{stabilizer group} of $R$ is defined by
$$
\cns{Stab}(R) =\{u \in U(n) \mid u R u^\dagger = R\},
$$
\item the \emph{automorphism group} \cite{zbMATH07502493} of $R$ is defined by
$$
\cns{Aut}(R) = \{\CC u \mid u \in \cns{Stab}(R)\},
$$
\item $R$ is defined to be \emph{vertex-transitive} if
$$
\sum_{F \in \cns{Aut}(R)} F = M_n(\CC).
$$
\end{enumerate}
\end{definition}

Of course, $\cns{Aut}(R) \iso \mathrm{Stab}(R)/ \TT 1_n$. The virtue of \cref{defn. aut}(2) is that the elements of $\cns{Aut}(R)$ are subspaces of $M_n(\CC)$, i.e., quantum relations on $M_n(\CC)$. \cref{defn. aut}(3) is the direct quantum analogue of vertex-transitivity for graphs; see section~\ref{sec. quantum graphs}. Equivalently, $R$ is vertex-transitive if $\cns{Stab}(R)' = \CC 1_n$.

\cref{defn. aut}(3) is stated for $\dagger$-closed subspaces because this case quickly reduces to the case of quantum graphs as the following proposition shows.

\begin{proposition}\label{reflexive or irreflexive}
Let $R \leq M_n(\CC)$ be such that $R^\dagger  = R$. If $R$ is vertex-transitive, then either $R$ is a quantum graph or $1_n \in R$.
\end{proposition}

\begin{proof}
Assume that $R$ is vertex-transitive, and let $\mu$ be normalized Haar measure on $\cns{Stab}(R)$. Assume that $R$ is not a quantum graph, and let $a \in R$ be such that $\tr(a) = 1$. Let
$c = \int v a v^\dagger\, d\mu(v) \in R$. We calculate that
\begin{align*}&
u c u^\dagger = \int u v a (u v)^\dagger \, d\mu(v) = \int v a v^\dagger \, d\mu(v) = c 
\end{align*}
for all $u \in \cns{Stab}(R)$ and that
\begin{align*}&
\tr(c) = \int \tr(v a v^\dagger)\, d\mu(v) = \int \tr(a) \, d\mu(v) = \tr(a) = 1.
\end{align*}
Thus, $c \in \cns{Stab}(R)' = \CC 1_n$, so $c = 1_n$. 
\end{proof}

\begin{remark}\label{rem. all}
A subspace $R \leq M_n(\CC)$ such that $R^\dagger = R$ is analogous to a graph that may have loops \cite{arXiv:2601.09685}. Thus, \cref{reflexive or irreflexive} is analogous to the proposition that a vertex-transitive graph that may have loops either has no loops or has all loops.
It also implies that the classification in \cref{main} is, in effect, a classification of all vertex-transitive $R \leq M_3(\CC)$ such that $R^\dagger = R$. Indeed, such a subspace $R$ is either a quantum graph or is of the form $\CC 1_3 + S$, where $S \leq M_3(\CC)$ is a quantum graph.
\end{remark}

\begin{definition}[\cite{zbMATH07668023} and \cite{arXiv:2510.21503}]\label{defn. regular}
Let $R \leq M_n(\CC)$ be a quantum graph. Then,
\begin{enumerate}
\item the \emph{degree matrix} of $R$ is the Hermitian matrix
$$
\cns{deg}(R) =\sum_{i =1}^d a_i a_i^\dagger,
$$
where $\{a_1, \ldots, a_d\}$ is any orthonormal basis of $R$,
\item $R$ is \emph{regular} if $\cns{deg}(R) \in \CC 1_n$.
\end{enumerate}
\end{definition}

The degree matrix $\cns{deg}(R)$ in \cref{defn. regular}(1) does not depend on the choice of the orthonormal basis $\{a_1, \ldots, a_d\}$ for $R$. This can be proved by computing that $\deg(R)$ is the matrix of the operator $n^2 (\id \tensor \tr)([R])$, where $[R] \in L(M_n(\CC)) = L(\CC^n \tensor \CC^{n\dagger})$ is the orthogonal projection operator onto the subspace $R \leq M_n(\CC)$. When the basis elements $a_i$ are all Hermitian,
$$
\cns{deg}(R) = \sum_{i =1}^d a_i^2.
$$

The expression $n^2 (\id \tensor \tr)([R])$ for $\deg(R)$ also provides the motivation for \cref{defn. regular} because the projection operator $[R]$ is the quantum analogue of the adjacency matrix of a graph. The factor $n^2$ in this expression is present because, when we view a Hermitian matrix $a \in M_n(\CC)$ as a real-valued function on a quantum set, the quantity $n^2 \tr(a)$ is the sum of its values \cite[Remark~1.4]{zbMATH07990859}.

\begin{proposition}\label{vertex-transitive regular}
Let $R \leq M_n(\CC)$ be a quantum graph. If $R$ is vertex-transitive, then it is regular.
\end{proposition}

\begin{proof}
Assume that $R$ is vertex-transitive, and let $\{a_1, \ldots, a_d\}$ be a Hermitian orthonormal basis of $R$. We compute that, for each $u \in \cns{Stab}(R)$,
$$
u \cns{deg}(R) u^\dagger = u \left( \sum_{i = 1}^d a_i^2 \right) u^\dagger = \sum_{i = 1}^d (u a_i u^\dagger)^2  = \deg(R).
$$
Thus, $\deg(R)$ commutes with all matrices $u \in \Stab(R)$. By \cref{defn. aut}(3), it follows that $\deg(R)$ commutes with all matrices in $M_n(\CC)$, so $\cns{deg}(R) \in \CC 1_n$. Therefore, $R$ is regular.
\end{proof}

Classically, both vertex-transitivity and regularity are retained by graph complements. The same occurs in the quantum setting.

\begin{definition}\label{defn. complement}
Let $R \leq M_n(\CC)$ be a quantum graph. The \emph{complement} of $R$ is the quantum graph $\bar R \leq M_n(\CC)$ that is defined by
$$
\bar R = \{a \in M_n(\CC) \mid \tr(a) = 0,\, a \perp R   \}.
$$
\end{definition}

It is immediate from \cref{defn. complement} that $\bar R \iso \bar S$ iff $R \iso S$.

\begin{proposition}\label{complement}
Let $R \leq M_n(\CC)$ be a quantum graph. Then,
\begin{enumerate}
\item $\bar{\bar R} = R$,
\item $\Stab(\bar R) = \Stab(R)$,
\item $\Aut(\bar R) = \Aut(R)$,
\item $\bar R$ is vertex-transitive if $R$ is vertex-transitive,
\item $\bar R$ is regular if $R$ is regular.
\end{enumerate}
\end{proposition}

\begin{proof}
To prove claim~1, we observe that $\bar R = (\CC 1_n + R)^\perp$. To prove claim~2, we observe that, for each unitary $u \in M_n(\CC)$, the operator $\Ad(u) \: M_n(\CC) \to M_n(\CC)$ that maps $a$ to $u a u^\dagger$ is itself unitary. We conclude that $\Stab(\bar R) = \Stab(R)$ since the invariant subspaces of $\Ad(u)$ are closed under the ortholattice operations. Claims~3~and~4 follow immediately from claim~2.

To prove claim~5, assume that $R$ is regular. Let $\{a_1, \ldots, a_d\}$ be a Hermitian orthonormal basis of $R$, and let $\{a_{d+1}, \ldots, a_{n^2-1}\}$ be a Hermitian orthonormal basis of $\bar R$. Then, $\{1_n, a_1, \ldots, a_{n^2-1}\}$ is a Hermitian orthonormal basis of $M_n(\CC)$, so
$$
1_n + d 1_{n} + \sum_{i = d + 1}^{n^2-1} a_i^2
=
1_n + \sum_{i = 1}^{n^2-1} a_i^2 = \sum_{i,j = 1}^{n} e_{ij}e_{i j}^\dagger  = \sum_{i,j = 1}^{n} e_{ii} = n 1_n,
$$
which implies that $\cns{deg}(\bar R) = (n -d - 1) 1_n$. Therefore, $\bar R$ is regular.
\end{proof}

\begin{definition}
Let $R \leq M_n(\CC)$ be a quantum graph, and let $d = \dim(R)$. The \emph{panoramic polynomial} of $R$ is the polynomial
$$
p_R(t_1, \ldots, t_d) = \det (t_1 a_1 + \cdots + t_d a_d) \in \RR [t_1, \ldots, t_d],
$$
which is defined up to orthogonal equivalence. Recall that two polynomials $p_1, p_2 \in \RR[t_1, \ldots, t_d]$ are \emph{orthogonally equivalent} if there exists an orthogonal matrix $r \in O(d)$ such that $p_1(t) = p_2(rt)$.
\end{definition}

\begin{proposition}
Let $R, S \leq M_n(\CC)$ be isomorphic quantum graphs. Then, their panoramic polynomials $p_R$ and $p_S$ are orthogonally equivalent.
\end{proposition}

\begin{proof}
Let $u \in M_n(\CC)$ be a unitary matrix such that $uR u^\dagger = S$. Let $\{a_1, \ldots, a_d\}$ be a Hermitian orthonormal basis of $R$. It follows that $\{ua_1u^\dagger, \ldots, ua_d u^\dagger\}$ is a Hermitian orthonormal basis for $S$. Therefore,
\begin{align*}
p_S(t_1, \ldots, t_d) 
& =
\det (t_1 u a_1 u^\dagger + \cdots + t_d u a_d u^\dagger) 
\\ & =
\det (u (t_1 a_1 + \cdots + t_d a_d )u^\dagger)
\\ &= 
\det (t_1 a_1 + \cdots + t_d a_d) = p_R(t_1, \ldots, t_d). \tag*{\qedsymbol}
\end{align*}
\renewcommand{\qed}{}
\end{proof}

Thus, the panoramic polynomial of a quantum graph $R \leq M_n(\CC)$ is an isomorphism invariant.

\begin{definition}\label{defn. aut0}\label{defn. connected}
Let $R \leq M_n(\CC)$ be a quantum graph. Then,
\begin{enumerate}
\item for each $u \in \cns{Stab}(R)$, we define $\cns{Ad}_R(u): R \to R$ by $a \mapsto u a u^\dagger$,
\item we define $\cns{Aut}_0(R) = \{\cns{Ad}_R(u) \mid u \in \cns{Stab}(R)\}$,
\item we define $R$ to be \emph{connected} \cite{zbMATH07309764} if $R'' = M_n(\CC)$,
\item if $R$ is connected, then we define the \emph{diameter} of $R$ by
$$
\cns{diam}(R) = \cns{min} \{m \in \NN \mid (\CC 1_n + R)^m = M_n(\CC)\},
$$
and otherwise, we define $\cns{diam}(R) = \infty$.
\end{enumerate}
\end{definition}

For each $u \in \cns{Stab}(R)$, $\cns{Ad}_R(u)$ is a unitary operator $R \to R$. Furthermore, if $a \in M_n(\CC)$ is an eigenvector of $\cns{Ad}_R(u)$ with eigenvalue $\lambda$, then $a^\dagger$ is an eigenvector of $\cns{Ad}_R(u)$ with eigenvalue $\overline \lambda$, so the characteristic polynomial of $\cns{Ad}_R(u)$ has real coefficients.

The following proposition shows that $\cns{Aut}_0(R)$ is always isomorphic to a closed subgroup of $O(p_R)$, which is the orthogonal symmetry group of the panoramic polynomial $p_R$. For every map $f\: \RR^d \to \RR$, we define $O(f)$ to be the set of all orthogonal matrices $r \in O(d)$ such that $f(rt) = f(t)$ for all $t \in \RR^d$. For every subset $A \subseteq \RR^d$, we define $O(A)$ similarly.

\begin{proposition}\label{kernel}
Let $R \leq M_n(\CC)$ be a quantum graph, and let $p_R$ be the panoramic polynomial of $R$ for some Hermitian orthonormal basis $\{a_1, \ldots, a_d\}$. Let $M_R$ be the subset of $S^{d-1} \subseteq \RR^d$ on which $p_R$ takes on its maximum value. We have the group homomorphisms
$$
\begin{tikzcd}
\Stab(R)
\arrow{r}{\varphi_1}
&
\Aut(R)
\arrow{r}{\varphi_2}
&
\Aut_0(R)
\arrow{r}{\varphi_3}
&
O(p_R)
\arrow{r}{\varphi_4}
&
O(M_R),
\end{tikzcd}
$$
where $\varphi_1\: u \mapsto \CC u$, $\varphi_2\: \CC u \mapsto \Ad_R(u)$, and $\varphi_3 \: \Ad_R(u) \mapsto  \iota^{-1} \circ \Ad_R(u) \circ \iota$ with $\iota\: \RR^d \to R$ mapping $(\alpha_1, \ldots, \alpha_d)$ to $\alpha_1 a_1 + \cdots + \alpha_d a_d$.
\begin{enumerate}
\item The kernel of $\varphi_1$ is $\TT 1_n$.
\item The kernel of $\varphi_2$ is $\{\CC u \mid u \in R' \cap U(n)\}$,
\item The kernel of $\varphi_3$ is zero,
\item The kernel of $\varphi_4$ is zero if $M_R$ spans $\RR^d$.
\end{enumerate}
Thus, if $R$ is connected, then $\CC u \mapsto \iota^{-1} \circ \Ad_R(u) \circ \iota$ is an injective group homomorphism $\Aut(R) \to O(p_R).$
\end{proposition}

\begin{proof}
Claim~1 is immediate because $\TT$ is the group of unitaries in $\CC$. The group homomorphism $\varphi_2$ is well defined by claim~1. Its kernel is the group $\{\CC u \mid u \in R' \cap U(n)\}$ because $R' \cap U(n) \subseteq \Stab(R)$ and $\Ad_R(u)\:R \to R$ is the identity map iff $u \in R'$.

The range of the group homomorphism $\varphi_3$ in claim~3 is in $O(p_R)$ because $$p_R(\alpha_1, \ldots, \alpha_d) = \det(\iota(\alpha_1, \ldots, \alpha_d))$$
for all $\alpha_1, \ldots, \alpha_d \in \RR$, which implies that
\begin{align*}
p_R((\iota^{-1} \circ \Ad_R(u) \circ \iota)(\alpha_1, \ldots, \alpha_d))
& =
\det((\Ad_R(u) \circ\iota)(\alpha_1, \ldots, \alpha_d))
\\ &=
\det(u \iota(\alpha_1, \ldots, \alpha_d) u^\dagger)
\\ &=
\det ( \iota( \alpha_1, \ldots, \alpha_d))
=
p_R(\alpha_1, \ldots, \alpha_d).
\end{align*}
The map $\varphi_3$ is an injective group homomorphism because $\iota$ is invertible.

Claim~4 holds because any linear transformation is uniquely determined by its values on a spanning subset of its domain.

If $R$ is connected, then $R' = \CC 1_n$ by \cref{defn. connected}(3), so $\varphi_2$ and $\varphi_3$ are injective by claims 2 and 3, respectively.
\end{proof}

When the image of $u \in \cns{Stab}(R)$ is an element $g$ of $\Aut(R)$, $\Aut_0(R)$, $O(p_R)$, or $O(M_R)$ via \cref{kernel}, we say that $u$ \emph{implements} the group element $g$.

\begin{example}
The quantum graphs $R \leq M_2(\CC)$ are classified modulo isomorphism by \cite[Theorem~3.11]{zbMATH07632578} in terms of the Pauli matrices 
$$
\hat x =
\begin{bmatrix}
0 & 1 \\
1 & 0 
\end{bmatrix},
\qquad
\hat y =
\begin{bmatrix}
0 & -i \\
i & 0 
\end{bmatrix},
\qquad
\hat z =
\begin{bmatrix}
1 & 0 \\
0 & -1 
\end{bmatrix}.
$$
These quantum graphs and their panoramic polynomials are shown in Figure~\ref{fig. VT2}.

\begin{figure}[ht]
\begin{center}
\begin{tabular}{| c |c|c|c| c | } 
 \hline
 $R$ & \text{orthogonal basis} & \text{panoramic polynomial} & $\Aut(R)$ & $\cns{diam}(R)$ \\
 \hline
 $\mathrm{VT}^2_0$ & $\varnothing$ & 0 & $SO(3)$ & $\infty$ \\
 $\mathrm{VT}^2_1$ &  $\{\hat x\}$ & $- t_1^2$ & $O(2)$ & $\infty$ \\ 
 $\mathrm{VT}^2_2$ & $\{\hat x, \hat y\}$ & $- t_1^2 - t_2^2 $ & $O(2)$  & $2$  \\
 $\mathrm{VT}^2_3$ & $\{\hat x, \hat y, \hat z\}$ & $-t_1^2 -t_2^2 - t_3^2$ & $SO(3)$ & $1$  \\ 
 \hline
\end{tabular}
\end{center}
\caption{The quantum graphs in $M_2(\CC)$ modulo isomorphism.}
\label{fig. VT2}
\end{figure}
\end{example}

The quantum graphs in $M_2(\CC)$ are all vertex-transitive and, hence, regular. We now turn to the classification of regular quantum graphs in $M_3(\CC)$. Clearly the quantum graph $\VT^3_0 = \{0\} \leq M_3(\CC)$ is vertex-transitive and, hence, regular. It is also easy to show that there are no regular quantum graphs $R \leq M_3(\CC)$ with $\dim(R) = 1$.

\begin{proposition}\label{dim R neq 1}
Let $R \leq M_n(\CC)$ be a regular quantum graph, and assume that $n$ is odd. Then, $\dim(R) \neq 1$.
\end{proposition}

\begin{proof}
Suppose that $\dim(R) = 1$. Then, $R$ has a Hermitian orthonormal basis that consists of a single Hermitian matrix $a \in R$ such that $a^2 = 1_n$. It follows that the eigenvalues of $a$ are in $\{1, -1\}$, which contradicts that $a$ is traceless. Therefore, $\dim(R) \neq 1$.
\end{proof}

\begin{lemma}\label{singular}
Let $R \leq M_n(\CC)$ be such that $R^\dagger = R$, and assume that $n$ is odd. If $\dim(R) \geq 2$, then $R$ contains a nonzero singular Hermitian matrix.
\end{lemma}

\begin{proof}
Assume that $\dim(R) \geq 2$, and suppose that every nonzero Hermitian matrix in $R$ is invertible. Then, $R$ has a Hermitian orthonormal basis consisting of invertible matrices $a_1, \ldots, a_d  \in R$. Let $p(t) \in \RR[t]$ be the polynomial $p(t) = \det(a_1 + a_2t)$. Diagonalizing $a_2$, we observe that the leading term of $p(t)$ is $\det(a_2) t^n$, so $p(t)$ has odd degree. Furthermore, $p(t)$ has real coefficients because it is real-valued. Thus, $p(t)$ has a root in $\RR$. We find that $R$ contains a singular Hermitian matrix of the form $a_1 + t a_2$, which is nonzero because $a_1$ and $a_2$ are linearly independent. Therefore, if $\dim(R) \geq 2$, then $R$ does contain a nonzero singular Hermitian matrix.
\end{proof}

The assumption that $n$ is odd is necessary, as the quantum graph $\mathrm{VT}^2_2$ demonstrates. We apply \cref{singular} to show that there is a unique regular quantum graph $R \leq M_3(\CC)$ such that $\dim(R) = 2$ modulo isomorphism. We provide a Hermitian orthonormal basis for this quantum graph and compute its panoramic polynomial in this basis. We formulate other such results similarly. 

\begin{proposition}\label{VT32}
Let $R \leq M_3(\CC)$ be a quantum graph such that $\dim(R) = 2$. If $R$ is regular, then it is isomorphic to the regular quantum graph
$$
\mathrm{VT}^3_2 =
\CC
\sqrt{\frac 1 2}
\begin{bmatrix}
-1 & 0 & 0 \\
0 & -1 & 0 \\
0 & 0 & 2
\end{bmatrix}
+
\CC
\sqrt{\frac 3 2}
\begin{bmatrix}
-1 & 0 & 0 \\
0 & 1 & 0 \\
0 & 0 & 0 
\end{bmatrix}
$$
and has panoramic polynomial $p_R(t_1, t_2) = \frac 1 {\sqrt 2}(t_1^3 - 3 t_1 t_2^2)$.
\end{proposition}

\begin{proof}
Assume that $R$ is regular. By \cref{singular}, it has a Hermitian orthonormal basis consisting of matrices $a, b \in R$ such that $a$ is singular. Since $a$ is singular, traceless, and norm-one, without loss of generality,
$$
a =
\sqrt{\frac 3 2}
\begin{bmatrix}
- 1 & 0 & 0 \\
0 & 1 & 0 \\
0 & 0 & 0 
\end{bmatrix}.
$$
It follows that $b$ is a Hermitian matrix that is norm-one and that satisfies
$$
b^2 = 
\frac 1 2
\begin{bmatrix}
1 & 0 & 0 \\
0 & 1 & 0 \\
0 & 0 & 4
\end{bmatrix},
$$
so without loss of generality,
$
b = \sqrt{\frac 1 2}
\left[
\begin{smallmatrix}
c &  0 \\
0  & 2
\end{smallmatrix}
\right],
$
where $c \in M_2(\CC)$ is a Hermitian matrix such that $c^2 = 1_2$. Since $b$ is traceless, $\cns{Tr}(c) = - 2$, so $c = - 1_2$. Thus,
$$
b = \sqrt{\frac 1 2}
\begin{bmatrix}
- 1 & 0 & 0 \\
0 & -1 & 0 \\
0 & 0 & 2
\end{bmatrix}.
$$
\end{proof}

This presentation of $\mathrm{VT}_2^3$ in terms of a Hermitian orthonormal basis yields its regularity and its panoramic polynomial immediately. It also has a basis-free presentation as
$$
\mathrm{VT}^3_2 = \{a \in M_3 (\CC) \mid \tr(a) = 0,\, a\text{ is diagonal}\}.
$$

\begin{proposition}\label{Aut VT32}
The quantum graph $\cns{VT}^3_2$ is vertex-transitive. Its automorphism group is
$$
\Aut(\cns{VT}^3_2) \iso \TT^2 \rtimes S_3.
$$
\end{proposition}

\begin{proof}
The quantum graph $\cns{VT}^3_2$ is vertex-transitive because $\Stab(\VT^3_2)$ contains the diagonal unitaries and the permutation unitaries.

We calculate $\Aut(\cns{VT}^3_2)$ using \cref{kernel}. By \cref{VT32}, the panoramic polynomial of this quantum graph is $p(t_1, t_2) = \frac 1 {\sqrt 2}(t_1^3 - 3 t_1 t_2^2)$, which achieves its maximum value on the set $$M = \{(1,0), (- 1/2,\sqrt 3 /2), ( - 1/2, - \sqrt 3 /2)\} \subseteq S^1.$$
Each permutation of $M$ is implemented by a permutation matrix in $\Stab(\VT^3_2)$, so $\Aut_0(\cns{VT}^3_2) \iso S_3$. Separately, we calculate that the kernel of the group homomorphism $\Aut(\cns{VT}^3_2) \to \Aut_0(\cns{VT}^3_2)$ is isomorphic to $$\{\CC u \mid u \in (\cns{VT}^3_2)' \cap U(3)\} = \{\CC u \mid u \in \CC^3 \cap U(3)\} \iso \TT^2.$$
Finally, this group homomorphism splits because each element of the codomain $\Aut_0(\cns{VT}^3_2) \iso S^3$ is implemented by a permutation matrix in $\Stab(\cns{VT}^3_2)$.
\end{proof}

\section{$3$-dimensional quantum graphs in $M_3(\CC)$}

In this section, we classify the regular quantum graphs $R \leq M_3(\CC)$ such that $\cns{dim}(R) = 3$ and observe that these quantum graphs are all vertex-transitive.

\begin{lemma}\label{singulars}
Let $R \leq M_n(\CC)$ be such that $R^\dagger = R$, and assume that $n$ is odd.
If $\dim(R) \geq 2$, then $R$ has a Hermitian orthonormal basis such that at most one basis element is invertible.
\end{lemma}

\begin{proof}
This proof is by induction on $\dim(R)$. The base case $\dim(R) = 2$ follows from \cref{singular}. For the induction step, assume that every $\dagger$-closed subspace $S \leq M_n(\CC)$ with $2 \leq \dim(S) < \dim(R)$ has a Hermitian orthonormal basis such that at most one basis element is invertible. By \cref{singular}, $R$ contains a norm-one singular Hermitian matrix $a_d$. The subspace $S = \{a \in R \mid a \perp a_d\}$ is clearly $\dagger$-closed, so by the induction hypothesis, $S$ has a Hermitian orthonormal basis $\{a_1, \ldots, a_{d-1}\}$ such that $a_i$ is singular for all $1 < i < d$. We conclude that $\{a_1, \ldots, a_d\}$ is a Hermitian orthonormal basis of $R$ such that $a_i$ is singular for all $1 < i \leq d$.
\end{proof}

\begin{lemma}\label{squares}
Let $R \leq M_3(\CC)$ be a quantum graph such that $\dim(R) = 3$. If $R$ is regular, then $R$ is isomorphic to the span of norm-one Hermitian matrices $a_1, a_2, a_3$ that satisfy
\begin{equation*}
a_1^2=\frac{3}{2}\begin{bmatrix}0 & 0 & 0\\ 0 & 1 & 0\\ 0 & 0 & 1\end{bmatrix}, 
\quad a_2^2=\frac{3}{2}\begin{bmatrix}1 & 0 & 0\\ 0 & 0 & 0\\ 0 & 0 & 1\end{bmatrix},
\quad a_3^2=\frac{3}{2}\begin{bmatrix}1 & 0 & 0\\ 0 & 1 & 0\\ 0 & 0 & 0\end{bmatrix}.
\end{equation*}
\end{lemma}

\begin{proof}
Assume that $R$ is regular. By \cref{singular}, $R$ has a Hermitian orthonormal basis consisting of matrices $a_1, a_2, a_3 \in R$ such that $a_2$ and $a_3$ are singular. Since $a_2$ and $a_3$ are singular, traceless, and norm-one, we have that
$$
a_2 \sim a_3 \sim
\sqrt{\frac 3 2}
\begin{bmatrix}
-1 & 0 & 0 \\
0 & 1 & 0 \\
0 & 0 & 0 
\end{bmatrix}, \quad \qquad
a_2^2 \sim a_3^2 \sim
\frac 3 2
\begin{bmatrix}
1 & 0 & 0 \\
0 & 1 & 0 \\
0 & 0 & 0
\end{bmatrix},
$$
where $\sim$ denotes unitary equivalence.
We infer that $a_2^2 = \frac 3 2 (1_3 - x_2 x_2^\dagger)$ and $a_3^2 = \frac 3 2 (1_3 - x_3 x_3^\dagger)$ for some unit vectors $x_2, x_3 \in \CC^3$. The equation $a_1^2 + a_2^2 + a_3^2 = 31_3$ now implies that
$$
a_1^2 = \frac 3 2 x_2 x_2^\dagger + \frac 3 2 x_3 x_3^\dagger.
$$
Hence, $a_1^2$ is singular, and so is $a_1$. As before, we deduce the existence of a unit vector $x_1 \in \CC^3$ such that $a_1^2 = \frac 3 2 (1_3 - x_1 x_1^\dagger)$. 

We now have unit vectors $x_1, x_2, x_3 \in \CC^3$ such that $a_i^2 = \frac 3 2 (1_3 - x_i x_i^\dagger)$ for all $i \in \{1, 2, 3\}$. The equation $a_1^2 + a_2^2 + a_3^2 = 3 1_3$ now yields
$$
1_3 = x_1 x_1^\dagger + x_2 x_2^\dagger + x_3 x_3^\dagger.
$$
It follows that $\{x_1, x_2, x_3\}$ is an orthonormal basis of $\CC^3$. Conjugating $a_1$, $a_2$, and $a_3$ by the unitary $[\begin{matrix} x_1 & x_2 & x_3 \end{matrix}]^\dagger$, we conclude that $R$ is isomorphic to a quantum graph with a Hermitian orthonormal basis of the claimed form.
\end{proof}

\begin{theorem}\label{VT33}
Let $R \leq M_3(\CC)$ be a quantum graph with $\dim(R) = 3$. If $R$ is regular, then $R$ is isomorphic to the regular quantum graph
\begin{equation*}
\mathrm{VT}^3_3(\theta) =
\CC \sqrt{\frac 3 2}
\begin{bmatrix}
0 & 0 & 0 \\
0 & 0 & 1 \\
0 & 1 & 0
\end{bmatrix}
+
\CC \sqrt{\frac 3 2}
\begin{bmatrix}
0 & 0 & 1 \\
0 & 0 & 0 \\
1 & 0 & 0
\end{bmatrix}
+
\CC \sqrt{\frac 3 2}
\begin{bmatrix}
0 & e^{-i\theta} & 0 \\
e^{i \theta} & 0 & 0 \\
0 & 0 & 0
\end{bmatrix}
\end{equation*}
for some $\theta \in \RR$ and has panoramic polynomial $p_R(t_1, t_2, t_3) = 3 \sqrt{\frac 3 2} \cos (\theta) t_1 t_2 t_3$.
\end{theorem}

\begin{proof}
By \cref{squares}, we may assume without loss of generality that $R$ has a Hermitian orthonormal basis consisting of matrices $a_1, a_2, a_3 \in R$ such that $a_i^2 e_i = 0$ for all $i \in \{1, 2, 3\}$. Since these matrices are Hermitian, $a_i e_i = 0$ for all $i \in \{1, 2, 3\}$. Since these matrices are also traceless, we find that
\begin{equation*}
a_1=
\sqrt{\frac{3}{2}}
\begin{bmatrix}0 & 0 & 0 \\ 
0 & \alpha_1 & \overline \beta_1 \\ 
0 & \beta_1 & -\alpha_1\end{bmatrix}, 
\quad \qquad
a_2=
\sqrt{\frac{3}{2}}
\begin{bmatrix}- \alpha_2 & 0 & \beta_2 \\
0 & 0 & 0 \\
\overline \beta_2 & 0 & \alpha_2\end{bmatrix},
\end{equation*}
\begin{equation*}
a_3= 
\sqrt{\frac{3}{2}}
\begin{bmatrix}\alpha_3 & \overline \beta_3 & 0 \\
\beta_3 & -\alpha_3 & 0 \\
0 & 0 & 0\end{bmatrix},
\end{equation*}
for some $\alpha_1, \alpha_2, \alpha_3 \in \RR$ and $\beta_1, \beta_2, \beta_3 \in \CC$. Since these matrices are pairwise orthogonal, $\alpha_i \neq 0$ for at most one $i \in \{1,2,3\}$. We may assume without loss of generality that $\alpha_1 = \alpha_2 = 0$. Because $\|a_1\| = \|a_2\| = 1$, we infer that $|\beta_1| = |\beta_2| = 1$. Conjugating by the diagonal unitary $\cns{diag}(\overline \beta_2, \beta_1, 1)$, we may assume without loss of generality that $\beta_1 = \beta_2 = 1$. Thus,
\begin{equation*}
R \iso
\CC \sqrt{\frac 3 2}
\begin{bmatrix}
0 & 0 & 0 \\
0 & 0 & 1 \\
0 & 1 & 0
\end{bmatrix}
+
\CC \sqrt{\frac{3}{2}}
\begin{bmatrix}
0 & 0 & 1\\
0 & 0 & 0 \\
1 & 0 & 0
\end{bmatrix}
+
\CC \sqrt{\frac{3}{2}}
\begin{bmatrix}
\alpha & \overline \beta & 0 \\
\beta & -\alpha & 0 \\
0 & 0 & 0
\end{bmatrix}
\end{equation*}
for some $\alpha \in \RR$ and $\beta \in \CC$.

For each angle $\varphi \in \RR$, conjugation by the unitary matrix
$$
u_\varphi = 
\begin{bmatrix}
\cos (\varphi) & - \sin(\varphi) & 0 \\
\sin(\varphi) & \cos(\varphi) & 0 \\
0 & 0 & 1
\end{bmatrix}
$$
leaves the subspace $\CC a_1 + \CC a_2 \leq R$ invariant. Furthermore, writing $a_3$ as a real-linear combination of the extended Pauli matrices $\sigma_1 \oplus 0, \sigma_2 \oplus 0, \sigma_3 \oplus 0 \in M_3(\CC)$, we find that conjugation by $u_\varphi$ implements a rotation by angle $2 \varphi$ about the $(\sigma_2 \oplus 0)$-axis. Therefore, there exists an angle $\varphi \in \RR$ such that $u_\varphi a_3 u_\varphi^\dagger$ is orthogonal to $\sigma_3 \oplus 0$. We conclude that
\begin{equation*}
R \iso
\CC \sqrt{\frac 3 2}
\begin{bmatrix}
0 & 0 & 0 \\
0 & 0 & 1 \\
0 & 1 & 0
\end{bmatrix}
+
\CC \sqrt{\frac{3}{2}}
\begin{bmatrix}
0 & 0 & 1\\
0 & 0 & 0 \\
1 & 0 & 0
\end{bmatrix}
+
\CC \sqrt{\frac{3}{2}}
\begin{bmatrix}
0 & \overline \beta & 0 \\
\beta & 0 & 0 \\
0 & 0 & 0
\end{bmatrix}
\end{equation*}
for some $\beta \in \CC$. Without loss of generality, $\beta = e^{i \theta}$ for some angle $\theta \in \RR$, and this choice makes the third generator norm-one.
\end{proof}

The following proposition yields a classification of three-dimensional regular quantum graphs in $M_3(\CC)$ modulo isomorphism by the parameter $\theta$.

\begin{proposition}\label{theta1 theta2}
Let $\theta_1, \theta_2 \in \RR$. Then, the following are equivalent:
\begin{enumerate}
\item $\cns{VT}^3_3(\theta_1) \iso \cns{VT}^3_3(\theta_2)$,
\item $\theta_1 + \theta_2 \in \pi \ZZ$ or $\theta_1 - \theta_2 \in \pi \ZZ$.
\end{enumerate}
\end{proposition}

\begin{proof}
Assume that $\cns{VT}^3_3(\theta_1) \iso \cns{VT}_3^3(\theta_2)$. It follows that their panoramic polynomials $3 \sqrt{\frac 3 2} \cos (\theta_1) t_1 t_2 t_3$ and $3 \sqrt{\frac 3 2} \cos (\theta_2) t_1 t_2 t_3$ are orthogonally equivalent, which implies that the maxima of these polynomials on the unit sphere are equal. In other words, $\frac 1 {\sqrt 2} \cos (\theta_1) = \frac 1 {\sqrt 2} \cos (\theta_2)$. Therefore, $\theta_1 + \theta_2 \in \pi \ZZ$ or $\theta_1 - \theta_2 \in \pi \ZZ$.

Conversely, assume that $\theta_1 + \theta_2 \in \pi \ZZ$ or $\theta_1 - \theta_2 \in \pi \ZZ$. For all $\theta \in \RR$, we have $\cns{VT}^3_3 (- \theta) = u \cns{VT}^3_3(\theta) u^\dagger$ and $\cns{VT}^3_3 (\theta + \pi) = v \cns{VT}^3_3(\theta) v^\dagger$, where
$$
u =
\begin{bmatrix}
0 & 1 & 0 \\
1 & 0 & 0 \\
0 & 0 & 1
\end{bmatrix},
\quad \qquad
v =
\begin{bmatrix}
-1 & 0 & 0 \\
0 & 1 & 0 \\
0 & 0 & 1
\end{bmatrix}.
$$
Hence, we may assume without loss of generality that $\theta_1 - \theta_2 \in 2 \pi \ZZ$, which implies that $e^{i \theta_1} = e^{i \theta_2}$. Therefore, $\cns{VT}^3_3(\theta_1) \iso \cns{VT}^3_3(\theta_2)$.
\end{proof}

\begin{theorem}\label{Aut VT33}
For each $\theta \in [0, \pi/2]$, the quantum graph $\cns{VT}^3_3(\theta)$ is vertex-transitive and connected. Furthermore,
\begin{enumerate}
\item $\Aut(\cns{VT}^3_3(0)) \iso S_4$,
\item $\Aut(\cns{VT}^3_3(\theta)) \iso A_4$ for $0 < \theta < \pi/2$,
\item $\Aut(\cns{VT}^3_3(\pi/2)) \iso SO(3)$.
\end{enumerate}
\end{theorem}

\begin{proof}
Let $\theta \in [0, \pi/2]$. The quantum graph $\cns{VT}^3_3(\theta)$ is vertex-transitive because the matrices
$$
e^{-i\theta/3}
\begin{bmatrix}
0 & 0 & 1 \\
e^{i \theta} & 0 & 0 \\
0 & 1 & 0 
\end{bmatrix},
\begin{bmatrix}
-1 & 0 & 0\\
0 & 1 & 0\\
0 & 0 & 1
\end{bmatrix}
\in \Stab(\VT^3_3(\theta))
$$
generate a subgroup of $U(3)$ that spans $M_3(\CC)$. Similarly, the defining Hermitian orthonormal basis of $\VT^3_3(\theta)$ generates $M_3(\CC)$ as a unital $\dagger$-algebra, so $\VT^3_3(\theta)$ is connected.

For $\theta \in [0, \pi/2)$, the set $M_\theta \subseteq S^2$ on which the panoramic polynomial $p_\theta(t_1, t_2, t_3) = 3 \sqrt {\frac 3 2} \cos(\theta) t_1 t_2 t_3$ attains its maximum value is $$M_\theta = \{x_0, x_1, x_2, x_3\},$$ 
$$
x_0 = \frac 1 {\sqrt 3}(1,1,1), \qquad \quad x_1 = \frac 1 {\sqrt 3} (1,-1,-1),
$$
$$
x_2 = \frac 1 {\sqrt 3}(-1, 1, -1), \qquad \quad x_3 =\frac 1 {\sqrt 3}(-1, -1, 1).
$$
By \cref{kernel}, $\Aut(\VT^3_3(\theta))$ is isomorphic to the subgroup of $O(M_\theta) \iso S_4$ that is implemented by $\Stab(\VT^3_3(\theta))$.

If $\theta = 0$, then $\Aut(\VT^3_3(\theta)) \iso S_4$, because the matrices
$$
\begin{bmatrix}
0 & 1 & 0 \\
1 & 0 & 0 \\
0 & 0 & 1
\end{bmatrix},
\begin{bmatrix}
0 & 0 & 1 \\
1 & 0 & 0 \\
0 & 1 & 0
\end{bmatrix},
\begin{bmatrix}
- 1 & 0 & 0 \\
0 & 1 & 0 \\
0 & 0 & 1
\end{bmatrix}
\in \Stab(\VT^3_3(0))
$$
implement the permutations $(x_1\; x_2)$, $(x_1\; x_2 \; x_3)$ and $(x_0\; x_1) (x_2\; x_3)$ of $M_0$, respectively, and together, these permutations generate $S_4$.

Assume $0 < \theta < \pi/2$. The matrices
$$
e^{-i\theta/3}
\begin{bmatrix}
0 & 0 & 1 \\
e^{i \theta} & 0 & 0 \\
0 & 1 & 0 
\end{bmatrix},
\begin{bmatrix}
- 1 & 0 & 0 \\
0 & 1 & 0 \\
0 & 0 & 1
\end{bmatrix}
\in \Stab(\VT^3_3(\theta))
$$
implement permutations $(x_1\; x_2 \; x_3)$ and $(x_0\; x_1) (x_2\; x_3)$ of $M_\theta$, respectively, and together these permutations generate $A_4 \subseteq S_4$. To show that $\Stab(\VT^3_3(\theta))$ does not implement $S_4$, it is sufficient to find a single permutation of $M_\theta$ that $\Stab(\VT^3_3(\theta))$ does not implement.

Suppose that there is a unitary $u \in \Stab(\VT^3_3(\theta))$ that implements the permutation $(x_1\; x_2)$. By \cref{kernel}, it commutes with both
$$
\begin{bmatrix}
0 &e^{-i\theta} & 1 \\
e^{i \theta} & 0 & 1\\
1 & 1 & 0
\end{bmatrix},
\begin{bmatrix}
0 &e^{-i\theta} & -1 \\
e^{i \theta} & 0 & -1\\
-1 & -1 & 0
\end{bmatrix}
\in \VT^3_3(\theta),
$$
which together generate $M_3(\CC)$ as a unital $\dagger$-algebra because they have no common eigenvector. Thus, $u \in \CC 1_3$. This conclusion contradicts our choice of $u$. Hence, there does not exist a unitary $u \in \Stab(\VT^3_3(\theta))$ that implements the permutation $(x_1\; x_2)$. Therefore, $\Aut(\VT^3_3(\theta)) \iso A_4$.

The remaining case is $\theta = \pi /2$. We compute the panoramic polynomial of the quantum graph
\begin{equation*}
R =
\CC \sqrt{\frac 3 2}
\begin{bmatrix}
0 & 0 & 0 \\
0 & 0 & -i \\
0 & i & 0
\end{bmatrix}
+
\CC \sqrt{\frac 3 2}
\begin{bmatrix}
0 & 0 & i \\
0 & 0 & 0 \\
-i & 0 & 0
\end{bmatrix}
+
\CC \sqrt{\frac 3 2}
\begin{bmatrix}
0 & -i & 0 \\
i & 0 & 0 \\
0 & 0 & 0
\end{bmatrix},
\end{equation*}
and find that it is zero. By \cref{VT33}, $R \iso \VT^3_3(\pi/2)$, so $\Aut(\VT^3_3(\pi/2)) \iso \Aut(R)$. Further, by \cref{kernel}, $\Aut(R)$ is isomorphic to the subgroup of $O(M_{\pi/2}) \iso O(3)$ that is implemented by $\Stab(R)$. 

The real subspace $R_h$ that consists of Hermitian matrices in $R$ is the Lie algebra $\mathfrak{so}(3)$. For $u \in SO(3) \subseteq \Stab(R)$, the map $u \mapsto \Ad_R(u)$ is the adjoint representation of $SO(3)$, so $\Stab(R)$ implements the subgroup $SO(3) \subseteq O(3)$. To show that $\Stab(R)$ does not implement $O(3)$, it is sufficient to find a single orthogonal transformation that $\Stab(R)$ does not implement.

Suppose that there is a unitary $u \in \Stab(R)$ that implements inversion through the origin. This unitary $u$ anticommutes with each of the defining basis elements of $R$ and, hence, commutes with each elementary matrix $e_{ij}$ for $i \neq j$. Thus, $u \in \CC 1_3$. This conclusion contradicts our choice of $u$. Hence, there does not exist a unitary $u \in \Stab(R)$ that implements inversion through the origin. Therefore $\Aut(\VT^3_3(\pi/2)) \iso \Aut(R) \iso SO(3)$.
\end{proof}

\section{$4$-dimensional quantum graphs in $M_3(\CC)$}

In this section, we establish that there exists a unique vertex-transitive quantum graph $R \leq M_3(\CC)$ such that $\cns{dim}(R) = 4$. We leave the existence of a regular quantum graph $R \leq M_3(\CC)$ that is not vertex-transitive unresolved.

\begin{proposition}\label{regular dim4 is connected}
Let $R \leq M_3(\CC)$ be a regular quantum graph. If $\dim(R) = 4$, then $R$ is connected.
\end{proposition}

\begin{proof}
Assume that $\dim(R) = 4$, and suppose that $R$ is not connected. It follows that $R'' \iso M_2(\CC) \oplus \CC$. We may assume without loss of generality that $R'' = M_2(\CC) \oplus \CC$. Since $R$ is a four-dimensional quantum graph such that $R \leq R''$, it consists of exactly the traceless matrices in $R''$. This quantum graph is not regular, contradicting our choice of $R$.
\end{proof}

\begin{proposition}\label{no nilpotents}
Let $R \leq M_3(\CC)$ be a quantum graph such that $\dim(R) = 4$. If $R$ is regular, then $a^2 = 0$ implies $a = 0$ for all $a \in R$.
\end{proposition}

\begin{proof}
Assume that $R$ is regular. Let $a \in R$, and assume $a^2 =0$. It follows that the range of $a$ is a subspace of the kernel of $a$, so the rank of $a$ is at most $1$. Suppose that $\rank(a) = 1$. We may assume without loss of generality that $a = \psi_1 \psi_2^\dagger$ for some unit vectors $\psi_1,\psi_2 \in \CC^3$ such that $\psi_1^\dagger \psi_2 = 0$.

The matrices $$\textstyle a_1 = \sqrt{\frac 3 2} (\psi_1 \psi_2^\dagger + \psi_2 \psi_1^\dagger), \qquad \quad a_2 = \sqrt{\frac 3 2} (i \psi_1 \psi_2^\dagger - i \psi_2 \psi_1^\dagger)$$ are Hermitian and norm-one, and both are clearly elements of $R$. We extend this pair to a Hermitian orthonormal basis $(a_1, a_2, a_3, a_4)$ of $R$. We may assume without loss of generality that $a_3$ is singular by applying \cref{singular} to the quantum graph $\CC a_3 + \CC a_4$.

The regularity of $R$ implies that 
$$
3 \psi_1 \psi_1^\dagger + 3 \psi_2 \psi_2^\dagger + a_3^2 + a_4^2 = 4 1_3,
$$
$$
a_3^2 + a_4^2 = \psi_1 \psi_1^\dagger + \psi_2 \psi_2^\dagger + 4 \psi_3 \psi_3^\dagger,
$$
where $(\psi_1, \psi_2, \psi_3)$ is an orthonormal basis of $\CC^3$. By a standard optimization argument, $\|b\|_\infty \leq \sqrt 2$ for norm-one Hermitian traceless matrices $b \in M_3(\CC)$ with equality iff $\cns{Spec}(b)$ is $\{-1, -1, 2\}$ or $\{1, 1, -2\}$ with multiplicity. It follows that the equation $\|a_3^2 + a_4^2\|_\infty = 4$ contradicts our assumption that $a_3$ is singular. Therefore, $\rank(a) =0$.
\end{proof}

In this section, $\omega = - \frac 1 2 + i \frac {\sqrt 3} 2$ denotes the standard primitive third root of unity. Further, $\hat u$ and $\hat v$ denote the corresponding shift and clock matrices
$$
\hat u =
\begin{bmatrix}
0 & 0 & 1 \\
1 & 0 & 0 \\
0 & 1 & 0
\end{bmatrix},
\quad \qquad
\hat v = 
\begin{bmatrix}
1 & 0 & 0 \\
0 & \omega & 0 \\
0 & 0 & \omega^2
\end{bmatrix},
$$
which are unitary matrices that satisfy $\hat u^3 = 1_3 = \hat v^3$ and $\hat v \hat u = \omega \hat u  \hat v$. The corresponding discrete Fourier transform matrix is the unitary
$$
\hat w  =
\frac 1 {\sqrt 3}
\begin{bmatrix}
1 & 1 & 1 \\
1 & \omega & \omega^2 \\
1 & \omega^2 & \omega
\end{bmatrix}.
$$
It satisfies $\hat w \hat u \hat w^\dagger = \hat v$, $\hat w \hat v \hat w^\dagger = \hat u^2$, and $\hat w^4 = 1_3$. Thus, we have
$$
\CC \hat u + \CC \hat u^2 \iso \CC \hat v + \CC \hat v^2  = \cns{VT}^3_2.
$$

\begin{proposition}\label{VT34}
The quantum graph $$\cns{VT}^3_4 = \CC \hat u + \CC \hat u^2 + \CC \hat v + \CC \hat v^2$$ is vertex-transitive and connected. Its panoramic polynomial is $$p (t_1, t_2, t_3, t_4) = \frac 1 {\sqrt 2}(t_1^3 - 3 t_1 t_2^2 + t_3^3 - 3  t_3 t_4^2).$$
Furthermore, $\Stab(\cns{VT}^3_4)$ is generated by $\hat u$, $\hat v$, $\hat w$, and the unitary scalars, and $$\Aut(\cns{VT}^3_4) \iso (S_3 \wr S_2) \cap A_6 \leq S_6.$$
In other words, $\Aut(\cns{VT}^3_4)$ is isomorphic to the group of automorphisms of the disjoint union graph $C_3 + C_3$ that are even as permutations.
\end{proposition}

\begin{proof}
The unitaries $\hat u$ and $\hat v$ generate $M_3(\CC)$ as a unital $\dagger$-algebra. Since $\hat u, \hat v \in \VT^3_4$, this implies that $\VT^3_4$ is connected. Since $\hat u, \hat v \in \Stab(\VT^3_4)$, this implies that $\VT^3_4$ is vertex-transitive. In fact, we will show that $\hat u, \hat v, \hat w \in \Stab(\VT^3_4)$ implement generators of $\Aut(\VT^3_4)$.

The quantum graph $\VT^3_4$ can be presented as
$$
\VT^3_4 = \CC \frac 1 {\sqrt 2}(\hat u + \hat u^2) + \CC \frac 1 {i \sqrt 2} (\hat u - \hat u^2) + \CC \frac 1 {\sqrt 2}(\hat v + \hat v^2) + \CC \frac 1 {i \sqrt 2} (\hat v - \hat v^2).
$$
In this Hermitian orthonormal basis, the panoramic polynomial of $\VT^3_4$ can be calculated to be
$p (t_1, t_2, t_3, t_4) = \frac 1 {\sqrt 2}(t_1^3 - 3 t_1 t_2^2 + t_3^3 - 3  t_3 t_4^2).$ The set $M \subseteq S^3$ on which the panoramic polynomial attains its maximal value is
$$
M = \{x_1, x_2, x_3, x_4, x_5, x_6\},
$$
$$
x_1 = (1,0,0,0), \qquad x_2 = \frac 1 2 (-1, \sqrt 3, 0 ,0), \qquad x_3 = \frac 1 2 (-1, -\sqrt 3, 0, 0),
$$
$$
x_4 = (0, 0, 1, 0), \qquad x_5 = \frac 1 2 (0 ,0, -1, \sqrt 3), \qquad x_6 = \frac 1 2 (0, 0, -1, - \sqrt 3).
$$
By \cref{kernel}, $\Aut(\VT^3_4)$ is isomorphic to the subgroup of $O(M) \iso S_3 \wr S_2$ that is implemented by $\Stab(\VT^3_4)$. The unitaries $\hat u, \hat v, \hat w \in \Stab(\VT^3_4)$ implement the permutations $(x_4\; x_5\; x_6), (x_1\; x_3 \;x_2), (x_1\; x_4)(x_2\; x_5 \; x_3\; x_6) \in  S_6,$ respectively, which generate the subgroup of even permutations in $S_3 \wr S_2 \leq S_6$. To show that $\Stab(\VT^3_4)$ does not implement $S_3 \wr S_2$, it is sufficient to find a single element of $S_3 \wr S_2$ that $\Stab(\VT^3_4)$ does not implement.

Suppose that there is a unitary $u \in \Stab(\VT^3_4)$ that implements the permutation $(x_2 \; x_3)$. By \cref{kernel}, the unitary $u$ commutes with 
$$
\hat u + \hat u^2,\, \hat v + \hat v^2,\, \omega^2 \hat v + \omega \hat v^2,\, \omega \hat v + \omega^2 \hat v ^2 \in \VT^3_4.
$$
In particular, $u$ is a diagonal unitary that commutes with $\hat u + \hat u^2$, so $u \in \TT 1_3$. This conclusion contradicts our choice of $u$. Hence, there does not exist a unitary $u \in \Stab(\VT^3_4)$ that implements the permutation $(x_2\; x_3)$. Therefore, $\Aut(\VT^3_4) \iso (S_3 \wr S_2) \cap A_6$. Since $\hat u$, $\hat v$, and $\hat w$ implement generators in $O(M)$, they also implement generators of $\Aut(\VT^3_4)$, so $\Stab(\VT^3_4)$ is generated by $\hat u$, $\hat v$, $\hat w$, and the scalar unitaries by \cref{kernel}.
\end{proof}

\begin{theorem}\label{dim2 dim4}
Let $S \leq R \leq M_3(\CC)$ be regular quantum graphs. If $\dim(S) = 2$ and $\dim(R) = 4$, then $S \iso \cns{VT}^3_2$ and $R \iso \cns{VT}^3_4$.
\end{theorem}

\begin{proof}
Assume that $\dim(S) = 2$ and $\dim(R) = 4$. We immediately infer that $S \iso \cns{VT}^3_2$ by \cref{VT32}. To show that $R \iso \cns{VT}^3_4$, we first extend a Hermitian orthonormal basis $\{a_1, a_2\}$ of $S$ to a Hermitian orthonormal basis $\{a_1, a_2, a_3, a_4\}$ of $R$. Because both $S$ and $R$ are regular, $a_1^2 + a_2^2 = 2 1_3$ and $a_1^2 + a_2^2 + a_3^2 + a_4^2 = 4 1_3$, so $a_3^2 + a_4^2 = 2 1_3$. Thus, $T = \CC a_3 + \CC a_4$ is a regular quantum graph such that $S \perp T$ and $S + T = R$. In other words, $\CC 1_3 + S$ and $\CC 1_3 + T$ are quasiorthogonal \cite{zbMATH08105222, zbMATH05724597, zbMATH05258254} unital $\dagger$-subalgebras of $M_3(\CC)$.

By \cref{VT32}, $T \iso \cns{VT}^3_2 \iso S$, so there exists a unitary $w \in U(3)$ such that $T = w S w^\dagger$ and, hence, such that $\CC 1_3 + T = w( \CC 1_3 + S)w^\dagger$. We may assume without loss of generality that $S = \cns{VT}^3_2$, which implies that $\CC 1_3 + S$ is the unital $\dagger$-algebra of diagonal matrices in $M_3(\CC)$. By \cite[Proposition~2.2]{zbMATH08105222}, it follows that $\sqrt 3 w$ is a complex Hadamard matrix. By \cite[sec.~5]{zbMATH05070807}, it follows that $\sqrt 3 w$ is equivalent to $\sqrt 3 \hat w$ in the sense that $
\sqrt 3 w = d_1 p_1 (\sqrt 3 \hat w) p_2 d_2$ for some diagonal unitaries $d_1, d_2 \in M_3(\CC)$ and permutation matrices $p_1, p_2 \in M_3(\CC)$.

We conclude the proof by computing that
\begin{align*}
R & = S + T 
\\ & \iso
(\CC \hat v + \CC \hat v^2) + w (\CC \hat v + \CC \hat v^2) w^\dagger
\\ & =
(\CC \hat v + \CC \hat v^2) + d_1 p_1 \hat w p_2 d_2 (\CC \hat v + \CC \hat v^2) d_2^\dagger p_2^\dagger \hat w^\dagger p_1^\dagger d_1^\dagger
\\ & =
(\CC \hat v + \CC \hat v^2) + d_1 p_1 \hat w (\CC \hat v + \CC \hat v^2) \hat w^\dagger p_1^\dagger d_1^\dagger
\\ & = 
d_1 p_1 (\CC \hat v + \CC \hat v^2) p_1^\dagger d_1^\dagger + d_1 p_1 (\CC \hat u + \CC \hat u^2) p_1^\dagger d_1^\dagger
\\ & =
d_1 p_1(\CC \hat v + \CC \hat v^2 + \CC \hat u + \CC \hat u^2) p_1^\dagger d_1^\dagger,
\end{align*}
which implies that $R \iso \CC \hat v + \CC \hat v^2 + \CC \hat u + \CC \hat u^2$ by \cref{defn. iso}.
\end{proof}

\begin{corollary}\label{VT34 complement}
Let $R \leq M_3(\CC)$ be a quantum graph. If $R \iso \cns{VT}^3_4$, then $\bar R \iso \cns{VT}^3_4$.
\end{corollary}

\begin{proof}
We may assume without loss of generality that $R = \cns{VT}^3_4$. Its complement $\bar R$ is vertex-transitive by Propositions~\ref{complement}~and~\ref{VT34}, and $\dim(\bar R) = 4$. Furthermore, $\bar R$ contains the quantum graph $S = \CC \hat u \hat v + \CC \hat u^2 \hat v^2$, which is vertex-transitive because $\hat u, \hat v \in \cns{Stab}(\CC \hat u \hat v + \CC \hat u^2 \hat v^2)$. The quantum graphs $S$ and $R$ are regular by \cref{vertex-transitive regular}. We conclude by \cref{dim2 dim4} that $\bar R \iso \cns{VT}^3_4$.
\end{proof}

The operator $\Ad_{\VT^3_4}(\hat u \hat v) \in \Aut_0(\VT^3_4)$ has eigenvectors $\hat u$, $\hat u^2$, $\hat v$, and $\hat v^2$ with eigenvalues $\omega$, $\omega^2$, $\omega^2$, and $\omega$, respectively. The following lemma shows that if $R$ is a four-dimensional regular quantum graph, then the existence of an element in $\Aut_0(R)$ with this spectrum implies that $R \iso \VT^3_4$.

\begin{lemma}\label{1 in spectrum}
Let $R \leq M_3(\CC)$ be a regular quantum graph with $\dim(R) = 4$. If there exists a unitary $w \in \Stab(R)$ such that $1 \not \in \cns{Spec}(\Ad_R(w))$, then $R \iso \VT^3_4$.
\end{lemma}

\begin{proof}
Let $w \in \mathrm{Stab}(R)$, and assume that $1 \not \in \cns{Spec}(\cns{Ad}_R(w))$. Let $(a_1, a_2, a_3, a_4)$ be an orthonormal basis of eigenvectors of $\cns{Ad}_R(w)$ with corresponding eigenvalues $(\lambda_1, \lambda_2, \lambda_3, \lambda_4)$. We may assume without loss of generality that $w$ is diagonal. For each index $i \in \{1, 2, 3, 4\}$, each diagonal entry of $w a_i w^\dagger$ is equal to the corresponding diagonal entry of $a_i$ and also to the corresponding diagonal entry of $\lambda_i a_i$. Since $1 \not \in \cns{Spec}(\cns{Ad}_R(w))$, it follows that each diagonal entry of each basis element $a_i$ is zero. Thus, $R$ is orthogonal to the unital $\dagger$-algebra of diagonal matrices in $M_3(\CC)$. It follows that $\bar R$ is a four-dimensional regular quantum graph that contains $\cns{VT}^3_2$. We conclude that $\bar R \iso \cns{VT}^3_4$ by \cref{dim2 dim4} and then that $R \iso \cns{VT}^3_4$ by \cref{VT34 complement}.
\end{proof}

The operator $\Ad_{\VT^3_4}(\hat w) \in \Aut_0(\VT^3_4)$ has eigenvectors 
$$\frac 1 2 (\hat u + i^k \hat v^2 + i^{2k} \hat u^2 + i^{3k} \hat v),$$ for $k \in \{0,1,2,3\}$, with eigenvalues $i^k$. Thus, $\Spec(\Ad_{\VT^3_4}(\hat w)) = \{1, -1, i, -i\}$. The next lemma establishes that for any four-dimensional vertex-transitive quantum graph $R \leq M_3(\CC)$, there exists a unitary $w \in \Stab(R)$ such that $\Spec(\Ad_R(w)) = \{1, \pm1, \lambda, \overline \lambda \}$ with multiplicity.

\begin{lemma}\label{Spec AdR w}
Let $R \leq M_3(\CC)$ be a vertex-transitive quantum graph. If $\dim(R) = 4$, then there exists a unitary $w \in \cns{Stab}(R)$ such that $$\cns{Spec}(\cns{Ad}_R(w)) = \{1, \pm 1, \lambda, \overline \lambda\}$$ with multiplicity and with $\lambda \neq \overline \lambda \in \TT$.
\end{lemma}

\begin{proof}
If $R = \VT^3_4$, then $w = \hat w$ is such a unitary. Assume that $R \not \iso \VT^3_4$. Then, by \cref{1 in spectrum}, $1 \in \Spec(\Ad_R(w))$ for all $w \in \Stab(R)$.

Suppose $\cns{Spec}(\cns{Ad}_R(w)) \subseteq \{1,-1\}$ for all $w \in \cns{Stab}(R)$. It follows that every element of $\cns{Aut}_0(R)$ is order-two. However, $R$ is connected by \cref{regular dim4 is connected}, so $\cns{Aut_0}(R) \iso \cns{Aut}(R)$ by \cref{kernel}. Thus, every element of $\cns{Aut}(R)$ is order-two. In particular, $\cns{Aut}(R)$ is abelian.

We infer that for all $u, v \in \cns{Stab}(R)$, there exists a scalar $\alpha \in \TT$ such that $v u = \alpha u v$. We calculate that $v^2 u^2 = \alpha ^4 u^2 v^2 = \alpha^4 v^2 u^2$, concluding that $\alpha^4 = 1$. We also calculate that $\det(u) \det(v) = \det(v u) = \det (\alpha u v) = \alpha^3 \det(u) \det(v)$, concluding that $\alpha^3 = 1$. Thus, $\alpha = 1$. It follows that $\cns{Stab}(R)$ is abelian, contradicting that $R$ is vertex-transitive. Therefore, $\cns{Spec}(\cns{Ad}_R(w)) \not \subseteq \{1, -1\}$ for some $w \in \cns{Stab}(R)$.

Overall, we obtain a matrix $w \in \cns{Stab}(R)$ such that $1, \lambda, \overline \lambda \in \cns{Spec}(\cns{Ad}_R(w))$ for some $\lambda \not \in \RR$. Because the characteristic polynomial of $\cns{Ad}_R(w)$ has real coefficients, the fourth eigenvalue of $\cns{Ad}_R(w)$ must be real. We conclude that $\cns{Spec}(\cns{Ad}_R(w)) = \{1, \pm 1, \lambda, \overline \lambda\}$.
\end{proof}

If $w^2$ is \emph{nondegenerate} in the sense that it has three distinct eigenvalues, then $R$ must be isomorphic to $\VT^3_4$. Indeed, in this case, $$\Spec(\Ad_R(w^2)) = \{1, 1, \lambda^2, \overline \lambda^2\}$$ with multiplicity and with $\lambda^2 \neq 1$. Thus, the $1$-eigenspace of $\Ad_R(w^2)$ is a two-dimensional quantum graph that consists of matrices that commute with $w^2$ and is thus isomorphic to $\VT^3_2$. This implies that $R$ is isomorphic to $\VT^3_4$ by \cref{dim2 dim4}. Hence, the remainder of our proof that $\VT^3_4$ is the only four-dimensional vertex-transitive quantum graph in $M_3(\CC)$ is devoted to analyzing degenerate unitaries in $\Stab(R)$.

\begin{theorem}\label{four cases}
Let $R \leq M_3(\CC)$ be a quantum graph, and let $u \in \Stab(R)$. If $u$ is degenerate, then at least one of the following is true:
\begin{enumerate}
\item $u \in \TT 1_3$,
\item $R$ is not connected,
\item $a^2 = 0$ for some nonzero $a \in R$,
\item $\Spec(u) = \{\alpha, \alpha, -\alpha\}$ with multiplicity for some $\alpha \in \TT$.
\end{enumerate}
\end{theorem}

\begin{proof}
We may assume without loss of generality that $u = \diag(1,1, \gamma)$ for some $\gamma \in \TT$. If $\gamma = 1$, then we obtain case~1. If $\gamma = -1$, then we obtain case~4. Assume $\gamma \not \in \{1, -1\}$. If $R$ is orthogonal to the elementary matrices $e_{13}$ and $e_{23}$, then it is also orthogonal to the elementary matrices $e_{31}$ and $e_{32}$, so $R \leq M_2(\CC) \oplus \CC$, and we obtain case~2.

Assume that $R$ is not orthogonal to $e_{13}$ or not orthogonal to $e_{23}$. We may assume without loss of generality that $R$ is not orthogonal to $e_{13}$. It follows that one of the eigenvectors $a \in R$ of $\Ad_R(u)$ is not orthogonal to $e_{13}$. The operator $\Ad_R(u)$ is the restriction of
$$
\Ad(u)\: e_{ij}  \mapsto
\begin{cases}
\gamma e_{ij} & \text{if }e_{ij} \in \{e_{31}, e_{32}\},\\ 
\overline \gamma e_{ij} & \text{if }e_{ij}  \in \{e_{13}, e_{23}\},\\
e_{ij} & \text{otherwise}.
\end{cases}
$$
Since $\overline \gamma \neq \gamma$, we find that $a \in \CC e_{13} + \CC e_{23}$, obtaining case~3.
\end{proof}

\begin{corollary}\label{alpha alpha -alpha}
Let $R \leq M_3(\CC)$ be a regular quantum graph such that $\dim(R) = 4$, and let $w \in \Stab(R)$. If $w$ is degenerate and $w \not \in\TT 1_3$, then $\Spec(w) = \{\alpha, \alpha, -\alpha\}$ with multiplicity for some $\alpha \in \TT$.
\end{corollary}

\begin{proof}
Assume that $w$ is degenerate and that $w \not \in \TT 1_3$. We apply \cref{four cases} to $w$. Case~1 contradicts assumption. Case~2 contradicts \cref{regular dim4 is connected}. Case~3 contradicts \cref{no nilpotents}. The remaining case~4 is our conclusion.
\end{proof}

\begin{lemma}\label{1 i -i}
Let $R \leq M_3(\CC)$ be a regular quantum graph such that $\dim(R) = 4$. If $\Spec(u) = \{1, i, -i\}$ for some $u \in \Stab(R)$, then $R \iso \VT^3_4$.
\end{lemma}

\begin{proof}
Assume that there exists a unitary $u \in \Stab(R)$ such that $\Spec(u) = \{1, i, -i\}$. We may assume without loss of generality that $u = \diag(-i, 1, i)$. The operator $\Ad_R(u)$ is then the restriction of
$$
\Ad(u)\: e_{ij} \mapsto \begin{cases}
(+ i) e_{ij} & \text{if }e_{ij} \in \{e_{21}, e_{32}\},\\
(- i) e_{ij} & \text{if }e_{ij}  \in \{e_{12}, e_{23}\},\\
- e_{ij} & \text{if } e_{ij} \in \{ e_{13}, e_{31}\}, \\
e_{ij} & \text{otherwise},
\end{cases}
$$
where the factors $+i$ and $-i$ are square roots of $-1$. If the $1$-eigenspace of $\Ad_R(u)$ is two-dimensional, then $\VT^3_2 \leq R$ and $R \iso \VT^3_4$ by \cref{dim2 dim4}. Similarly, if the $1$-eigenspace of $\Ad_R(u)$ is zero-dimensional, then $\VT^3_2 \perp R$, and $R \iso \VT^3_4$ by \cref{dim2 dim4} and \cref{VT34 complement}.

Assume that the $1$-eigenspace of $\Ad_R(u)$ is one-dimensional. Certainly, the $(+i)$-eigenspace and $(-i)$-eigenspace of $\Ad_R(u)$ have equal dimensions because $\Ad_R(u)$ is a $\dagger$-map. These eigenspaces can neither both be zero-dimensional nor both be two-dimensional because the $1$-eigenspace of $\Ad_R(u)$ is one-dimensional. Hence, they are both one-dimensional. Therefore, $\Spec(\Ad_R(u)) = \{1, i, -1, -i\}$.

Choose norm-one eigenvectors $a_0, a_1, a_2, a_3 \in R$ for $\Ad_R(u)$ with eigenvalues $1, i, -1, -i \in \TT$, respectively. We may assume without loss of generality that
$$
a_0 = \sqrt{\frac 1 2}
\begin{bmatrix}
\alpha_0 & 0 & 0\\
0 & \beta_0 & 0 \\
0 & 0 & \gamma_0\\
\end{bmatrix}, 
\qquad \quad
a_2 =\sqrt {\frac 3 2}
\begin{bmatrix}
0 & 0 & \alpha_2 \\
0 & 0 & 0 \\
\overline \alpha_2 & 0 & 0
\end{bmatrix},
$$
$$
a_1 =
\sqrt 3
\begin{bmatrix}
0 & 0 & 0 \\
\alpha_1 & 0 & 0\\
0 & \beta_1 & 0
\end{bmatrix},
\qquad \quad
a_3 =
\sqrt 3
\begin{bmatrix}
0 & \overline \alpha_1 & 0 \\
0 & 0 & \overline \beta_1 \\
0 & 0 & 0
\end{bmatrix},
$$
for some $\alpha_0, \beta_0, \gamma_0 \in \RR$ and $\alpha_1, \beta_1, \alpha_2 \in \CC$. The condition that these matrices are norm-one implies that
\begin{enumerate}
\item $\alpha_0^2 + \beta_0^2 + \gamma_0^2 = 6$,
\item $|\alpha_1|^2 + |\beta_1|^2 = 1$,
\item $|\alpha_2|^2 = 1$.
\end{enumerate}
The condition that these matrices are traceless implies that $\alpha_0 + \beta_0 + \gamma _0 = 0$.

We may further assume without loss of generality that $0 \leq \alpha_1 \leq \beta_1 \in \RR$, that $\alpha_2 = 1$, and that $\alpha_0 \geq 0$ by conjugating $R$ by appropriate unitaries. We arrange for $\alpha_1, \beta_1 \geq 0$ by conjugating by $\diag (\alpha_1/|\alpha_1|, 1, \overline \beta_1 / |\beta_1|)$. Then, we arrange for $\alpha_1 \leq \beta_1$ by conjugating by a permutation matrix. Finally, we arrange for $\alpha_2 = 1$ by conjugating by $\diag(\overline \beta_2, 1, \beta_2)$, where $\beta_2$ is a square root of $\alpha_2$. We rename and rescale the matrices $a_i$, for $i \in \{0, 1, 2, 3\}$, as necessary. In particular, we arrange for $\alpha_0 \geq 0$ in the last step. The conditions that $0 \leq \alpha_1 \leq \beta_1 \in \RR$ and $|\alpha_1|^2 + |\beta_1|^2 = 1$ imply that $0 \leq \alpha_1 \leq 1 / \sqrt 2$.

We now appeal to the regularity of $R$ to compute that
\begin{align*}
\frac 1 2
\begin{bmatrix}
\alpha_0^2 & 0 & 0 \\
0 & \beta_0^2 & 0 \\
0 & 0 & \gamma_0^2
\end{bmatrix}
& =
a_0^2
=
41_3 - a_1^\dagger a_1 - a_2^\dagger a_2 - a_3^\dagger a_3
\\ & =
\frac 1 2
\begin{bmatrix}
5 - 6 \alpha_1^2 & 0 & 0 \\
0 & 2 & 0 \\
0 & 0 & 6 \alpha_1^2 - 1
\end{bmatrix},
\end{align*}
obtaining the system
$$
\begin{cases}
\alpha_0^2 = 5 - 6 \alpha_1^2 \\
\beta_0^2 = 2\\
\gamma_0^2 = 6 \alpha_1^2 -1 \\
\alpha_0^2 + \beta_0^2 + \gamma_0^2 = 6\\
\alpha_0 + \beta_0 + \gamma_0 = 0 \\
0 \leq \alpha_0 \\
0 \leq \alpha_1 \leq 1 / \sqrt 2.
\end{cases} 
$$
This system has the unique solution
$$
\begin{cases}
\alpha_0 = \sqrt{2 + \sqrt 3} \\
\beta_0 = - \sqrt 2\\
\gamma_0 = - \sqrt{2 - \sqrt 3}\\
\alpha_1 = \sqrt{3 - \sqrt 3} / \sqrt 6.
\end{cases}
$$
We then have $\beta_1 =\sqrt{1 - \alpha_1^2} =\sqrt{3 + \sqrt 3}/ \sqrt 6$ and $\alpha_2 = 1$. Therefore, every four-dimensional regular quantum graph that admits a unitary $u \in \Stab(R)$ with $\Spec(u) = \{1, i ,-i\}$ is isomorphic to $\CC a_0 +  \CC a_1 + \CC a_2 + \CC a_3$ for those specific values of the parameters $\alpha_0, \beta_0, \gamma_0, \alpha_1, \beta_1, \alpha_2 \in \RR$. 

The quantum graph $\VT^3_4$ is four-dimensional and regular by \cref{VT34} and \cref{vertex-transitive regular}, and $- i \hat w \in \Stab(\VT^3_4)$ with $\Spec(- i\hat w) = \{1, i, -i\}$. We conclude that every four-dimensional regular quantum graph $R$ that admits a unitary $u \in \Stab(R)$ with this spectrum is isomorphic to $\VT^3_4$.
\end{proof}

\begin{theorem}\label{theorem VT34}
Let $R \leq M_3(\CC)$ be a quantum graph with $\dim(R) = 4$. If $R$ is vertex-transitive, then $R$ is isomorphic to $\VT^3_4$.
\end{theorem}

\begin{proof}
Assume that $R$ is vertex-transitive. By \cref{Spec AdR w}, there exists a unitary $w \in \Stab(R)$ such that $\Spec(\Ad_R(w)) = \{1, \pm 1, \lambda, \overline \lambda\}$ with multiplicity and with $\lambda \neq \overline \lambda \in\TT$. It follows that $\Spec(\Ad_R(w^2)) = \{1, 1, \lambda^2, \overline {\lambda^2}\}$ with multiplicity. Hence, the $1$-eigenspace of $\Ad_R(w^2)$ is a two-dimensional quantum graph that is a subspace of the commutant $(\CC w^2)'$. Therefore, if $w^2$ is nondegenerate, then the $1$-eigenspace of $\Ad_R(w^2)$ is a quantum graph that is isomorphic to $\VT^3_2$ and is a subspace of $R$, so $R \iso \VT^3_4$ by \cref{dim2 dim4}.

Assume that $w^2$ is degenerate, and note that $w^2 \not \in \TT 1_3$. It follows that $\Spec(w^2) = \{\alpha^2, \alpha^2, -\alpha^2\}$ with multiplicity for some $\alpha \in \TT$ by \cref{alpha alpha -alpha}. Thus, $\Spec((\overline \alpha w)^2) = \{1, 1, -1\}$ with multiplicity, so $\Spec(\overline \alpha w) \subseteq \{1, -1, i, -i\}$. If $\overline \alpha w$ is nondegenerate, then $\Spec(i^k \overline \alpha w) = \{1, i, -i\}$ for some integer $k$, and $R \iso \VT^3_4$ by \cref{1 i -i}. If $\overline \alpha w$ is degenerate, then $\Spec( \overline \alpha w) = \{\beta, \beta, -\beta\}$ for some $\beta\in \TT$ by \cref{alpha alpha -alpha}, so $\Spec(\overline \beta \overline \alpha  w) = \{1, 1, -1\}$, contradicting $\lambda \in \Spec(\Ad_R(w)) = \Spec(\Ad_R(\overline \beta \overline \alpha  w))$. We conclude that $R \iso \VT^3_4$ in every case.
\end{proof}

A complete classification of vertex-transitive quantum graphs $R \leq M_3(\CC)$ can be obtained as a corollary of \cref{theorem VT34}. We isolate a part of this corollary as \cref{VT35} for easy reference.

\begin{definition}
For all $i, j \in \{1, 2, 3\}$ and all $\theta \in \RR$, let $\hat r_{ij}(\theta) \in M_3(\CC)$ be
$
\hat r_{i j}(\theta)= e^{-i \theta} e_{i j} + e^{i \theta} e_{ji}.
$
Let $\hat x_{i j} = \hat r_{i j} (0)$ and $\hat y_{i j} = \hat r_{i j} (\pi /2)$.
We define the following quantum graphs in $M_3(\CC)$:
\begin{enumerate}
\item $\VT^3_5(\theta) = \VT^3_2 + \CC \sqrt{ \frac 3 2} \hat x_{2 3} + \CC \sqrt{ \frac 3 2} \hat x_{3 1} + \CC \sqrt{ \frac 3 2} \hat r_{12}(\theta)$ for $\theta \in \RR$,
\item $\VT^3_6 = \CC \sqrt{ \frac 3 2} \hat x_{2 3} + \CC \sqrt{ \frac 3 2} \hat x_{3 1} + \CC \sqrt{ \frac 3 2} \hat x_{12} + \CC \sqrt{ \frac 3 2} \hat y_{2 3} + \CC \sqrt{ \frac 3 2} \hat y_{3 1} + \CC \sqrt{ \frac 3 2} \hat y_{12}$,
\item $\VT^3_8 = \VT^3_2 + \VT^3_6$.
\end{enumerate}
\end{definition}

\begin{proposition}\label{VT35}
Let $\theta \in [0, \pi/2]$. Then, the complement of $\VT^3_5(\theta)$ is isomorphic to $\VT^3_3(\frac \pi 2 - \theta)$. Hence, $\VT^3_5(\theta)$ is vertex-transitive.
\end{proposition}

\begin{proof}
The complement of $\VT^3_5(\theta)$ is the regular quantum graph
$$
R_\theta = \CC \hat y_{23} + \CC \hat y_{31} + \CC \hat r_{12}(\theta + \textstyle \frac \pi 2).
$$
It follows by \cref{complement} and by Theorems~\ref{VT33} and \ref{Aut VT33} that $\VT^3_5(\theta)$ is vertex-transitive. The panoramic polynomial of $R_\theta$ is $$\textstyle p(t_1, t_2, t_3) =3 \sqrt{\frac 3 2} \cos (\theta + \frac \pi 2) t_1 t_2 t_3 ,$$ so by \cref{VT33}, $R_\theta$ is isomorphic to $\VT^3_3(\theta + \frac \pi 2)$. Therefore, $R_\theta$ is isomorphic to $\VT^3_3(\frac \pi 2 - \theta)$ by \cref{theta1 theta2}.
\end{proof}

We now state the complete classification of vertex-transitive quantum graphs in $M_3(\CC)$. This classification is summarized in Figure \ref{fig. classification}. Recall that $\VT^3_0$ is the zero-dimensional quantum graph in $M_3(\CC)$.

\begin{corollary}\label{main}
The quantum graphs 
\begin{equation*}\label{classification list}\tag{$*$}
\VT^3_0, \VT^3_2, \VT^3_3(\theta), \VT^3_4, \VT^3_5(\theta), \VT^3_6, \VT^3_8 \leq M_3(\CC)
\end{equation*}
with $0 \leq \theta \leq \pi/2$, form a complete classification of vertex-transitive quantum graphs in $M_3(\CC)$ up to isomorphism. Explicitly,
\begin{enumerate}
\item each of the quantum graphs in eq.~(\ref{classification list}) is vertex-transitive,
\item no two of them are isomorphic, with $$\VT^3_3(\theta_1) \not \iso \VT^3_3(\theta_2), \qquad \quad \VT^3_5(\theta_1) \not \iso \VT^3_5(\theta_2)$$ when $\theta_1 \neq \theta_2$,
\item every vertex-transitive quantum graph $R \leq M_3(\CC)$ is isomorphic to a quantum graph in eq.~(\ref{classification list}).
\end{enumerate}
\end{corollary}

\begin{proof}
Claim~1.
The quantum graph $\VT^3_0$ is vertex-transitive simply because $\Stab(\VT^3_0) = U(3)$. The quantum graphs $\VT^3_2$, $\VT^3_3(\theta)$ and $\VT^3_4$ are all vertex-transitive by \cref{Aut VT32}, \cref{Aut VT33}, and \cref{VT34}, respectively. The quantum graphs $\VT^3_6$ and $\VT^3_8$ are vertex-transitive by \cref{complement} because they are the complements of $\VT^3_2$ and $\VT^3_0$, respectively. For each $\theta \in [0, \pi/2]$, the quantum graph $\VT^3_5(\theta)$ is vertex-transitive by \cref{VT35}.

Claim~2. The dimension of a quantum graph $R \leq M_3(\CC)$ is clearly an isomorphism invariant, so it remains to show that $\VT^3_3(\theta_1) \not \iso \VT^3_3(\theta_2)$ and $\VT^3_5(\theta_1) \not \iso \VT^3_5(\theta_2)$ when $\theta_1 \neq \theta_2$. The former claim follows by \cref{theta1 theta2}, and then the latter claim follows by \cref{VT35}.

Claim~3. Let $R\leq M_3(\CC)$ be a vertex-transitive quantum graph. If $\dim(R) = 0$ or $\dim(R) = 8$, then the claim is trivial. If $1 \leq \dim(R) \leq 4$, then the claim follows by Proposition \ref{dim R neq 1} or \ref{VT32} or by Theorem \ref{VT33} or \ref{theorem VT34}. If $5 \leq \dim(R) \leq 7$, then the claim follows by Propositions \ref{complement} and \ref{VT35}.
\end{proof}

\bibliographystyle{plain}
\bibliography{refs.bib}

\Addresses

\newpage

\begin{sidewaysfigure}[ht]
\centering
$
\quad
\hat u = e_{21} + e_{32} + e_{13}
\qquad
\hat v = e_{11} + \omega e_{22} + \omega^2 e_{33}
\qquad
\hat r_{jk}(\theta) = e^{i \theta} e_{kj} + e^{-i\theta}e_{jk}
\qquad
\hat x_{jk} = \hat r_{jk}(0)
\qquad
\hat y_{jk} = \hat r_{jk}( \frac \pi 2)
$

\vspace{2.5ex}

\centering

\begin{tabular}{|c|c|c|c|c|} 
 \hline
 $R$ & \text{orthogonal basis} & $ \sqrt 2(\text{panoramic polynomial})$ & $\Aut(R)$ & $\cns{diam}(R)$ \\
 \hline
 $\mathrm{VT}^3_0$ & $\varnothing$ & 0 & $PU(3)$ & $\infty$ \\
 $\mathrm{VT}^3_2$ & $\{\hat v,\hat v^2\}$ & $t_1^3 - 3t_1t_2^2$ & $\TT^2 \rtimes S_3$ & $\infty$ \\ 
 $\mathrm{VT}^3_3(0)$ & $\{\hat x_{12}, \hat x_{23}, \hat x_{31}\}$ & $3 \sqrt 3 t_1 t_2 t_3 $ & $S_4$  & $2$  \\
 $\mathrm{VT}^3_3(\theta)$ & $\{\hat x_{12},\hat x_{23},\hat r_{31}(\theta)\}$ & $3 \sqrt 3  \cos(\theta) t_1 t_2 t_3 $ & $A_4$  & $2$  \\
 $\mathrm{VT}^3_3(\frac \pi 2)$ & $\{ \hat x_{12} , \hat x_{23}, \hat y_{31}\}$ & $0$ & $SO(3)$  & $2$  \\
 $\mathrm{VT}^3_4$ & $\{\hat u, \hat u^2, \hat v , \hat v^2\}$ & $t_1^3 - 3t_1t_2^2 + t_3^3 - 3t_3t_4^2$ & $(S_3 \wr S_2) \cap A_6$ & $2$  \\
 $\mathrm{VT}^3_5(0)$ & $\{ \hat v , \hat v^2 , \hat x_{12} ,\hat x_{23} , \hat x_{31}\}$ & $t_1^3 - 3 t_1 t_2^2 + 3 \sqrt 3 t_3 t_4 t_5 + \frac 3 2 t_1 (t_3^2 + t_4^2 - 2 t_5^2) + \frac 3 2 \sqrt 3 t_2 (t_3^2 - t_4^2) $ & $SO(3)$  & $2$  \\
 $\VT^3_5(\theta)$ & $\{ \hat v , \hat v^2 , \hat x_{12} ,\hat x_{23} , \hat r_{31}(\theta)\}$ & $t_1^3 - 3 t_1 t_2^2 + 3 \sqrt 3 \cos(\theta) t_3 t_4 t_5 + \frac 3 2 t_1 (t_3^2 + t_4^2 - 2 t_5^2) + \frac 3 2 \sqrt 3 t_2 (t_3^2 - t_4^2) $ & $A_4$  & $2$  \\
 $\VT^3_5(\frac \pi 2)$ & $\{\hat v, \hat v^2, \hat x_{12}, \hat x_{23}, \hat y_{31} \}$ & $t_1^3 - 3 t_1 t_2^2 + \frac 3 2 t_1 (t_3^2 + t_4^2 - 2 t_5^2) + \frac 3 2 \sqrt 3 t_2 (t_3^2 - t_4^2) $ & $S_4$ & $2$ \\
 $\VT^3_6$ & $\{\hat x_{12}, \hat x_{23}, \hat x_{31}, \hat y_{12},  \hat y_{23}, \hat y_{31}\} $ & $3 \sqrt 3 (t_1 t_2 t_3 - t_1 t_5 t_6 - t_2 t_4 t_6 - t_3 t_4 t_5)$ &  $\TT^2 \rtimes S_3$ & 2 \\
 $\VT^3_8$ & $\{\hat v,  \hat v^2,  \hat x_{12}, \hat x_{23},  \hat x_{31}, \hat y_{12}, \hat y_{23}, \hat y_{31}\} $ & $p(t_1, \ldots, t_8)$ & $PU(3)$ & $1$ \\ 
 \hline
\end{tabular}

\vspace{1ex}

\centering
$
\qquad \qquad p(t_1, \ldots, t_8) = t_1^3 - 3 t_1 t_2^2 + 3 \sqrt 3(t_3 t_4 t_5 - t_3 t_7 t_8 - t_4 t_6 t_8 - t_5 t_6 t_7  ) + \frac 3 2 t_1 (t_3^2 + t_4^2 - 2t_5^2 + t_6^2+ t_7^2 - 2t_8^2 ) + \frac 3 2 \sqrt{3} t_2 (t_3^2 - t_4^2 + t_6^2 -t_7^2) 
$
\caption{The vertex-transitive quantum graphs in $M_3(\CC)$ modulo isomorphism.
}
\label{fig. classification}
\begin{minipage}{\textwidth}
\centering
(The given panoramic polynomials need not be computed in the given basis.)
\end{minipage}
\end{sidewaysfigure}

\end{document}